\documentclass[final,3p]{elsarticle}
\usepackage{lineno,hyperref}
\usepackage{amsmath}
\usepackage{amssymb}
\usepackage[linesnumbered,ruled,vlined]{algorithm2e}
\usepackage{caption}
\usepackage{mathtools}
\usepackage{changes}
\usepackage{multirow}
\usepackage{verbatim}
\usepackage{mathrsfs}
\usepackage{graphicx}
\usepackage{subcaption}
\usepackage{amsmath,amsthm,bm,mathrsfs}
\theoremstyle{plain}% Theorem-like structures provided by amsthm.sty
\newtheorem{theorem}{Theorem}[section]

\newtheorem{proposition}[theorem]{Proposition}

\theoremstyle{definition}

\theoremstyle{remark}
\newtheorem{remark}{Remark}

\modulolinenumbers[5]

\journal{ArXiv.org}

\begin{document}

\begin{frontmatter}

\title{A Low-rank ADI Algorithm for Solving Large-scale Non-symmetric Algebraic Riccati Equations}

\author[uz]{Umair~Zulfiqar\corref{mycorrespondingauthor}}
\cortext[mycorrespondingauthor]{Corresponding author}
\ead{umair@yangtzeu.edu.cn}
\address[uz]{School of Electronic Information and Electrical Engineering, Yangtze University, Jingzhou, Hubei, 434023, China}
\begin{abstract}
This paper considers large-scale nonsymmetric continuous-time algebraic Riccati equations (NAREs) that admit low-rank solutions. Low-rank alternating direction implicit (ADI) methods have proven to be an efficient approach for solving several matrix equations, including Lyapunov equations, Sylvester equations, and symmetric Riccati equations. Although a low-rank algorithm for the Sylvester equation has been used as an inner loop in computing low-rank solutions of NAREs, no low-rank ADI algorithm currently exists for NAREs themselves. This paper fills this gap by developing a low-rank ADI algorithm for large-scale NAREs that admit a low-rank solution. Since Lyapunov equations, Sylvester equations, and symmetric Riccati equations are special cases of the NARE, the existing low-rank ADI methods in the literature are special cases of the more general low-rank ADI method proposed here. An automatic and computationally efficient method for shift generation is also discussed, and a subspace-accelerated projection approach is presented to generate shifts for subsequent iterations without user intervention. Once initialized with arbitrary shifts, the proposed algorithm solves large-scale NAREs autonomously, generating its own shifts. Numerical results are presented using benchmark example of order $10^6$, demonstrating the computational efficiency and accuracy of the proposed algorithm.
\end{abstract}


\end{frontmatter}

%\linenumbers
\section{Introduction}
This paper considers the following nonsymmetric continuous-time algebraic Riccati equation (NARE). Let $X \in \mathbb{R}^{n \times \hat{n}}$ be a solution to
\begin{align}
A X \hat{E}+ E X \hat{A}-EX\hat{B}CX\hat{E}+B\hat{C}=0, \label{nare}
\end{align}
where $E, A \in \mathbb{R}^{n \times n}$, $B \in \mathbb{R}^{n \times m}$, $C \in \mathbb{R}^{p \times n}$, $\hat{E}, \hat{A} \in \mathbb{R}^{\hat{n} \times \hat{n}}$, $\hat{B} \in \mathbb{R}^{\hat{n} \times p}$, and $\hat{C} \in \mathbb{R}^{m \times \hat{n}}$. The dimensions $n$ and $\hat{n}$ are assumed to be large, and the matrices $A$, $E$, $\hat{A}$, and $\hat{E}$ are sparse.

Define the closed-loop matrices and associated gains as
\begin{align}
A_{\mathrm{cl}} &= E^{-1}\big(A-K_{\mathrm{gain}}C\big), \qquad K_{\mathrm{gain}} = EX\hat{B},\nonumber\\
\hat{A}_{\mathrm{cl}} &= \big(\hat{A}-\hat{B}\hat{K}_{\mathrm{gain}}\big)\hat{E}^{-1}, \qquad \hat{K}_{\mathrm{gain}} = CX\hat{E}.\nonumber
\end{align}

For a stabilizing solution $X$ of \eqref{nare}, the gains $K_{\mathrm{gain}}$ and $\hat{K}_{\mathrm{gain}}$ ensure that $A_{\mathrm{cl}}$ and $\hat{A}_{\mathrm{cl}}$ are Hurwitz; that is, all their eigenvalues lie in the open left half of the complex plane \cite{freiling2002survey,abou2003non,bini2008fast}.

Now define the matrix
\[
\mathcal{K}=\begin{bmatrix}E^{-1}A & E^{-1}B\hat{C}\hat{E}^{-1} \\ \hat{B}C & -\hat{A}\hat{E}^{-1}\end{bmatrix}.
\]
The stabilizing solution of \eqref{nare} satisfies
\[
\lambda_i(\mathcal{K}) = \lambda_i(A_{\mathrm{cl}}) \cup \lambda_i(-\hat{A}_{\mathrm{cl}}),
\]
where $\lambda_i(\cdot)$ denotes the eigenvalues of a matrix.

In the typical setting where $p \ll n$ and $m \ll n$, the solution $X$ is often of low numerical rank. This property enables accurate low-rank approximations when $n$ is large, as storing or computing the full matrix $X$ directly would be computationally prohibitive \cite{li2013solving,benner2016low,weng2021solving}. The present work focuses on computing the stabilizing low-rank solution of \eqref{nare}. Nevertheless, the proposed algorithm is general enough to produce other solutions as well if a stabilizing solution is not required.
\section{Main Work}
In this section, we exploit the Petrov–Galerkin projection-based interpolatory nature of the alternating direction implicit (ADI) methods to develop a low-rank approach for NAREs. We first establish the theoretical foundation of the proposed method, followed by a discussion of its algorithmic aspects. Two implementations are presented: one that utilizes the Sherman–Morrison–Woodbury (SMW) formula \cite{golub2013matrix} and one that avoids it. The automatic generation of shifts is also addressed.
\subsection{The Low-rank ADI Approach for NAREs}
Let \(\{\alpha_i\}_{i=1}^{k} \subset \mathbb{C}_{-}\) and \(\{\beta_i\}_{i=1}^{k} \subset \mathbb{C}_{-}\) be two sets of ADI shifts used to obtain a low-rank approximation \(X \approx \tilde{X} = V^{(i)} \bar{X}^{(i)} \big(\hat{W}^{(i)}\big)^\top\), where the matrices \(V^{(i)}\) and \(\hat{W}^{(i)}\) satisfy the properties
\begin{align}
\underset{i=1,\dots,k}{\text{span}}\left\{(-\alpha_i E - A)^{-1} B\right\} &\subset \mathrm{Ran}\big(V^{(i)}\big), \label{int_prop1}\\
\underset{i=1,\dots,k}{\text{span}}\left\{(-\overline{\beta}_i \hat{E}^\top - \hat{A}^\top)^{-1} \hat{C}^\top\right\} &\subset \mathrm{Ran}\big(\hat{W}^{(i)}\big). \label{int_prop2}
\end{align}
Define the residual \(R_x^{(i)}\) for the approximation \(X \approx \tilde{X}\) as
\[
R_x^{(i)} = A \tilde{X} \hat{E}+ E \tilde{X} \hat{A}-E\tilde{X}\hat{B}C\tilde{X}\hat{E}+B\hat{C}.
\]
Then, for some generally unknown matrices \(W^{(i)}\) and \(\hat{V}^{(i)}\) satisfying \((W^{(i)})^\top E V^{(i)} = I\) and \((\hat{W}^{(i)})^\top \hat{E} \hat{V}^{(i)} = I\), the residual obtained from the general low-rank ADI method satisfies the Petrov–Galerkin projection condition
\[
(W^{(i)})^\top R_x^{(i)} \hat{V}^{(i)} = 0.
\]

Assume that the shifts \(\alpha_i\) and \(\beta_i\) are grouped according to the following four cases, consistent with the grouping used in the ADI method for Sylvester equations \cite{benner2014computing}:
\begin{enumerate}
\item Case I: \(\mathrm{Im}(\alpha_i)=0\) and \(\mathrm{Im}(\beta_i)=0\).
\item Case II: \(\mathrm{Im}(\alpha_i)\neq 0\), \(\alpha_{i+1}=\overline{\alpha}_i\), \(\mathrm{Im}(\beta_i)\neq0\), and \(\beta_{i+1}=\overline{\beta}_i\).
\item Case III: \(\mathrm{Im}(\alpha_i)\neq 0\), \(\alpha_{i+1}=\overline{\alpha}_i\), \(\mathrm{Im}(\beta_i)=0\), and \(\mathrm{Im}(\beta_{i+1})=0\).
\item Case IV: \(\mathrm{Im}(\alpha_i)=0\), \(\mathrm{Im}(\alpha_{i+1})=0\), \(\mathrm{Im}(\beta_i)\neq0\), and \(\beta_{i+1}=\overline{\beta}_i\).
\end{enumerate}

Corresponding to this grouping, define the matrices \(s_v^{(i)}\), \(l_v^{(i)}\), \(s_w^{(i)}\), and \(l_w^{(i)}\) as follows:
\begin{enumerate}
\item Case I: 
\[
s_v^{(i)} = -\alpha_i I_m,\quad l_v^{(i)} = -I_m,\quad s_w^{(i)} = -\beta_i I_m,\quad l_w^{(i)} = -I_m.
\]
\item Case II: 
\begin{align}
s_v^{(i)} &= \begin{bmatrix}
-\mathrm{Re}(\alpha_i)I_m & -\mathrm{Im}(\alpha_i)I_m \\
\mathrm{Im}(\alpha_i)I_m & -\mathrm{Re}(\alpha_i)I_m
\end{bmatrix},&
l_v^{(i)} &= \begin{bmatrix} -I_m & 0 \end{bmatrix},\nonumber\\
s_w^{(i)} &= \begin{bmatrix}
-\mathrm{Re}(\beta_i)I_m & -\mathrm{Im}(\beta_i)I_m \\
\mathrm{Im}(\beta_i)I_m & -\mathrm{Re}(\beta_i)I_m
\end{bmatrix},&
l_w^{(i)} &= \begin{bmatrix} -I_m & 0 \end{bmatrix}.\nonumber
\end{align}
\item Case III: 
\begin{align}
s_v^{(i)} &= \begin{bmatrix}
-\mathrm{Re}(\alpha_i)I_m & -\mathrm{Im}(\alpha_i)I_m \\
\mathrm{Im}(\alpha_i)I_m & -\mathrm{Re}(\alpha_i)I_m
\end{bmatrix},&
l_v^{(i)} &= \begin{bmatrix} -I_m & 0 \end{bmatrix},\nonumber\\
s_w^{(i)} &= \begin{bmatrix}
-\beta_i I_m & I_m \\
0 & -\beta_{i+1} I_m
\end{bmatrix},&
l_w^{(i)} &= \begin{bmatrix} -I_m & 0 \end{bmatrix}.\nonumber
\end{align}
\item Case IV: 
\begin{align}
s_v^{(i)} &= \begin{bmatrix}
-\alpha_i I_m & I_m \\
0 & -\alpha_{i+1} I_m
\end{bmatrix},& 
l_v^{(i)} &= \begin{bmatrix} -I_m & 0 \end{bmatrix},\nonumber\\
s_w^{(i)} &= \begin{bmatrix}
-\mathrm{Re}(\beta_i)I_m & -\mathrm{Im}(\beta_i)I_m \\
\mathrm{Im}(\beta_i)I_m & -\mathrm{Re}(\beta_i)I_m
\end{bmatrix},&
l_w^{(i)} &= \begin{bmatrix} -I_m & 0 \end{bmatrix}.\nonumber
\end{align}
\end{enumerate}

Define the matrices \(S_v^{(i)}\), \(L_v^{(i)}\), \(S_w^{(i)}\), \(L_w^{(i)}\), \(V^{(i)}\), \(\hat{W}^{(i)}\), \(B_{\perp}^{(i)}\), \(\hat{C}_{\perp}^{(i)}\), and \(\bar{X}^{(i)}\) as follows:
\begin{align}
S_v^{(i)} &= \begin{bmatrix}
S_v^{(i-1)} & \bar{X}^{(i-1)}\Big((L_w^{(i-1)})^\top l_v^{(i)} + (\hat{W}^{(i-1)})^\top \hat{B} C v_i\Big) \\
0 & s_v^{(i)}
\end{bmatrix}, \quad 
L_v^{(i)} = \begin{bmatrix} L_v^{(i-1)} & l_v^{(i)} \end{bmatrix}, \nonumber \\[4pt]
S_w^{(i)} &= \begin{bmatrix}
S_w^{(i-1)} & (\bar{X}^{(i-1)})^\top\Big((L_v^{(i-1)})^\top l_w^{(i)} + (V^{(i-1)})^\top C^\top \hat{B}^\top w_i\Big) \\
0 & s_w^{(i)}
\end{bmatrix}, \quad 
L_w^{(i)} = \begin{bmatrix} L_w^{(i-1)} & l_w^{(i)} \end{bmatrix}, \nonumber \\[4pt]
V^{(i)} &= \begin{bmatrix} V^{(i-1)} & v_i \end{bmatrix}, \quad 
\hat{W}^{(i)} = \begin{bmatrix} \hat{W}^{(i-1)} & w_i \end{bmatrix}, \nonumber \\[4pt]
B_{\perp}^{(i)} &= B - E V^{(i)} \bar{X}^{(i)} (L_w^{(i)})^\top = B_{\perp}^{(i-1)} - E v_i x_i (l_w^{(i)})^\top, \nonumber \\[4pt]
\hat{C}_{\perp}^{(i)} &= \hat{C} - L_v^{(i)} \bar{X}^{(i)} (\hat{W}^{(i)})^\top \hat{E} = \hat{C}_{\perp}^{(i-1)} - l_v^{(i)} x_i (w_i)^\top \hat{E}, \nonumber \\[4pt]
\bar{X}^{(i)} &= \mathrm{blkdiag}\big( \bar{X}^{(i-1)},\, x_i \big),
\label{SvLv}
\end{align}
where \(v_i\), \(w_i\), and \(x_i\) satisfy the following Sylvester equations:
\begin{align}
\Big(A - E V^{(i-1)} \bar{X}^{(i-1)} (\hat{W}^{(i-1)})^\top \hat{B} C\Big) v_i - E v_i s_v^{(i)} + B_{\perp}^{(i-1)} l_v^{(i)} &= 0, \label{eq:vi} \\[4pt]
\Big(\hat{A}^\top - \hat{E}^\top \hat{W}^{(i-1)} (\bar{X}^{(i-1)})^\top (V^{(i-1)})^\top C^\top \hat{B}^\top \Big) w_i - \hat{E}^\top w_i s_w^{(i)} + \hat{C}_{\perp}^{(i-1)} l_w^{(i)} &= 0, \label{eq:wi} \\[4pt]
-(s_w^{(i)})^\top x_i^{-1} - x_i^{-1} s_v^{(i)} + (l_w^{(i)})^\top l_v^{(i)} + w_i^\top \hat{B} C v_i &= 0. \label{eq:xi}
\end{align}

From these definitions, it can be readily verified that \(V^{(i)}\), \(\hat{W}^{(i)}\), and \((\bar{X}^{(i)})^{-1}\) satisfy the following Sylvester equations:
\begin{align}
A V^{(i)} - E V^{(i)} S_v^{(i)} + B L_v^{(i)} &= 0, \label{Sylv_V} \\[4pt]
\hat{A}^\top \hat{W}^{(i)} - \hat{E}^\top \hat{W}^{(i)} S_w^{(i)} + \hat{C}^\top L_w^{(i)} &= 0, \label{Sylv_W} \\[4pt]
-(S_w^{(i)})^\top (\bar{X}^{(i)})^{-1} - (\bar{X}^{(i)})^{-1} S_v^{(i)} + (L_w^{(i)})^\top L_v^{(i)} + (\hat{W}^{(i)})^\top \hat{B} C V^{(i)} &= 0. \label{Sylv_X}
\end{align}
Due to the connection between Sylvester equations and rational Krylov subspaces established in \cite{gallivan2004sylvester}, the matrices \(V^{(i)}\) and \(\hat{W}^{(i)}\) satisfy the properties \eqref{int_prop1} and \eqref{int_prop2}.

Let us define \(G(s)\) and \(\hat{G}(s)\) as
\[
G(s) = C (sE - A)^{-1} B \qquad \text{and} \qquad \hat{G}(s) = \hat{C} (s\hat{E} - \hat{A})^{-1} \hat{B}.
\]
Construct the following \(im\)-th order reduced-order models of \(G(s)\) and \(\hat{G}(s)\):
\[
G_r^{(i)}(s) = C_r^{(i)} (sI - A_r^{(i)})^{-1} B_r^{(i)} \qquad \text{and} \qquad \hat{G}_r^{(i)}(s) = \hat{C}_r^{(i)} (sI - \hat{A}_r^{(i)})^{-1} \hat{B}_r^{(i)},
\]
where
\begin{align}
(W^{(i)})^\top E V^{(i)} &= I, \quad A_r^{(i)} = (W^{(i)})^\top A V^{(i)}, \quad B_r^{(i)} = (W^{(i)})^\top B, \quad C_r^{(i)} = C V^{(i)}, \nonumber \\
(\hat{W}^{(i)})^\top \hat{E} \hat{V}^{(i)} &= I, \quad \hat{A}_r^{(i)} = (\hat{W}^{(i)})^\top \hat{A} \hat{V}^{(i)}, \quad \hat{B}_r^{(i)} = (\hat{W}^{(i)})^\top \hat{B}, \quad \hat{C}_r^{(i)} = \hat{C} \hat{V}^{(i)}. \nonumber
\end{align}
By pre-multiplying \eqref{Sylv_V} by \((W^{(i)})^\top\) and post-multiplying \eqref{Sylv_W} by \((\hat{V}^{(i)})^\top\), we readily obtain
\[
A_r^{(i)} = S_v^{(i)} - B_r^{(i)} L_v^{(i)} \qquad \text{and} \qquad \hat{A}_r^{(i)} = (S_w^{(i)})^\top - (L_w^{(i)})^\top \hat{C}_r^{(i)}.
\]
Note that due to the block upper-triangular structure of \(S_v^{(i)}\) and \(S_w^{(i)}\), their eigenvalues are, respectively, \((-\alpha_1, \dots, -\alpha_i)\) and \((-\beta_1, \dots, -\beta_i)\), each with multiplicity \(m\). Owing to the connection between Sylvester equations and rational interpolation established in \cite{gallivan2004sylvester}, the following interpolation conditions hold:
\[
G(-\alpha_i) = G_r^{(i)}(-\alpha_i) \quad \text{and} \quad \hat{G}(-\overline{\beta}_i) = \hat{G}_r^{(i)}(-\overline{\beta}_i), \qquad i = 1, \dots, k.
\]
Furthermore, these interpolation conditions are independent of the specific choices of \(W^{(i)}\) and \(\hat{V}^{(i)}\); consequently, the parameters \(B_r^{(i)}\) and \(\hat{C}_r^{(i)}\) can be chosen arbitrarily without affecting the interpolation conditions. For further details, see \cite{wolfthesis,panzerthesis,astolfi2010model}.

The following theorem provides specific choices for \(B_r^{(i)}\) and \(\hat{C}_r^{(i)}\) that ensure \(\bar{X}^{(i)}\) is a stabilizing solution to the projected NARE
\begin{align}
A_r^{(i)} \bar{X}^{(i)} + \bar{X}^{(i)} \hat{A}_r^{(i)} - \bar{X}^{(i)} \hat{B}_r^{(i)} C_r^{(i)} \bar{X}^{(i)} + B_r^{(i)} \hat{C}_r^{(i)} = 0. \label{proj_nare}
\end{align}

\begin{theorem}\label{th1}
Let \(\{\alpha_i\}_{i=1}^{k}\subset\mathbb{C}_{-}\) and \(\{\beta_i\}_{i=1}^{k}\subset\mathbb{C}_{-}\) be given ADI shifts. Let \(V^{(i)}\) and \(\hat{W}^{(i)}\) solve the Sylvester equations \eqref{Sylv_V} and \eqref{Sylv_W}, respectively, with \(S_v^{(i)}\), \(L_v^{(i)}\), \(S_w^{(i)}\), \(L_w^{(i)}\), \(B_{\perp}^{(i)}\), \(\hat{C}_{\perp}^{(i)}\), and \(\bar{X}^{(i)}\) defined in \eqref{SvLv}. Assume that there exist matrices \(W^{(i)}\) and \(\hat{V}^{(i)}\) satisfying \((W^{(i)})^\top E V^{(i)} = I\) and \((\hat{W}^{(i)})^\top \hat{E} \hat{V}^{(i)} = I\), such that \(A_r^{(i)} = S_v^{(i)} - B_r^{(i)} L_v^{(i)}\) and \(\hat{A}_r^{(i)} = (S_w^{(i)})^\top - (L_w^{(i)})^\top \hat{C}_r^{(i)}\). When the free parameters \(B_r^{(i)}\) and \(\hat{C}_r^{(i)}\) are set to
\begin{align}
B_r^{(i)} &= \bar{X}^{(i)} (L_w^{(i)})^\top = 
\begin{bmatrix} b_1 \\ \vdots \\ b_i \end{bmatrix} = 
\begin{bmatrix} x_1 (l_w^{(1)})^\top \\ \vdots \\ x_i (l_w^{(i)})^\top \end{bmatrix}, \\[4pt]
\hat{C}_r^{(i)} &= L_v^{(i)} \bar{X}^{(i)} = 
\begin{bmatrix} \hat{c}_1 & \cdots & \hat{c}_i \end{bmatrix} = 
\begin{bmatrix} l_v^{(1)} x_1 & \cdots & l_v^{(i)} x_i \end{bmatrix},
\end{align}
the following statements hold:
\begin{enumerate}
    \item \(\bar{X}^{(i)}\) is a stabilizing solution to the projected NARE \eqref{proj_nare}.
    
    \item The residual \(R_x^{(i)}\) for the approximation \(X \approx V^{(i)} \bar{X}^{(i)} (\hat{W}^{(i)})^\top\) satisfies the Petrov–Galerkin projection condition \((W^{(i)})^\top R_x^{(i)} \hat{V}^{(i)} = 0\) and is given by 
    \[
    R_x^{(i)} = B_{\perp}^{(i)} \hat{C}_{\perp}^{(i)}.
    \]
    
    \item The gain matrices \(K_{\mathrm{gain}}\) and \(\hat{K}_{\mathrm{gain}}\) can be approximated recursively as
    \[
    K_{\mathrm{gain}} \approx \tilde{K}^{(i)} = \tilde{K}^{(i-1)} + E v_i x_i w_i^\top \hat{B}, \qquad
    \hat{K}_{\mathrm{gain}} \approx \bar{K}^{(i)} = \bar{K}^{(i-1)} + C v_i x_i w_i^\top \hat{E}.
    \]
    
    \item The matrices \(V^{(i)}\) and \(\hat{W}^{(i)}\) satisfy the Sylvester equations
    \begin{align}
    A V^{(i)} - E V^{(i)} A_r^{(i)} + B_{\perp}^{(i)} L_v^{(i)} &= 0, \label{Vb_Sylv} \\
    \hat{A}^\top \hat{W}^{(i)} - \hat{E}^\top \hat{W}^{(i)} (\hat{A}_r^{(i)})^\top + (\hat{C}_{\perp}^{(i)})^\top L_w^{(i)} &= 0. \label{Wc_Sylv}
    \end{align}
\end{enumerate}
\end{theorem}
\begin{proof}
\textbf{1. Stabilizing solution property:}\\
Consider the following:
\begin{align}
& A_r^{(i)}\bar{X}^{(i)} + \bar{X}^{(i)}\hat{A}_r^{(i)} - \bar{X}^{(i)}\hat{B}_r^{(i)}C_r^{(i)}\bar{X}^{(i)} + B_r^{(i)}\hat{C}_r^{(i)} \nonumber \\
&= \big(S_v^{(i)} - \bar{X}^{(i)}(L_w^{(i)})^\top L_v^{(i)}\big)\bar{X}^{(i)} + \bar{X}^{(i)}\big((S_w^{(i)})^\top - (L_w^{(i)})^\top L_v^{(i)}\bar{X}^{(i)}\big) \nonumber \\
&\hspace*{4cm} - \bar{X}^{(i)}\hat{B}_r^{(i)}C_r^{(i)}\bar{X}^{(i)} + \bar{X}^{(i)}(L_w^{(i)})^\top L_v^{(i)}\bar{X}^{(i)} \nonumber \\
&= -\bar{X}^{(i)}\Big(-(S_w^{(i)})^\top (\bar{X}^{(i)})^{-1} - (\bar{X}^{(i)})^{-1} S_v^{(i)} + (L_w^{(i)})^\top L_v^{(i)} + (\hat{W}^{(i)})^\top \hat{B} C V^{(i)}\Big)\bar{X}^{(i)} \nonumber \\
&= 0.\nonumber
\end{align}
Pre-multiplying \eqref{Sylv_X} by \(\bar{X}^{(i)}\) yields
\[
A_r^{(i)} - \bar{X}^{(i)}\hat{B}_r^{(i)}C_r^{(i)} = -\bar{X}^{(i)}(S_w^{(i)})^\top (\bar{X}^{(i)})^{-1}.
\]
Similarly, post-multiplying \eqref{Sylv_X} by \(\bar{X}^{(i)}\) gives
\[
\hat{A}_r^{(i)} - \hat{B}_r^{(i)}C_r^{(i)}\bar{X}^{(i)} = -(\bar{X}^{(i)})^{-1} S_v^{(i)} \bar{X}^{(i)}.
\]
Since the ADI shifts \(\alpha_i\) and \(\beta_i\) have negative real parts, the matrices \(-S_v^{(i)}\) and \(-(S_w^{(i)})^\top\) are Hurwitz. Consequently, \(\bar{X}^{(i)}\) is a stabilizing solution to \eqref{proj_nare}.

\textbf{2. Residual and Petrov–Galerkin condition:}\\
Observe that
\begin{align}
B_{\perp}^{(i)}\hat{C}_{\perp}^{(i)} &= \Big(B - E V^{(i)} \bar{X}^{(i)} (L_w^{(i)})^\top\Big) \Big(\hat{C} - L_v^{(i)} \bar{X}^{(i)} (\hat{W}^{(i)})^\top \hat{E}\Big) \nonumber \\
&= B\hat{C} - B L_v^{(i)} \bar{X}^{(i)} (\hat{W}^{(i)})^\top \hat{E} - E V^{(i)} \bar{X}^{(i)} (L_w^{(i)})^\top \hat{C} \nonumber \\
&\hspace*{4cm} + E V^{(i)} \bar{X}^{(i)} (L_w^{(i)})^\top L_v^{(i)} \bar{X}^{(i)} (\hat{W}^{(i)})^\top \hat{E}.\nonumber
\end{align}
Hence,
\[
B\hat{C} = B_{\perp}^{(i)}\hat{C}_{\perp}^{(i)} + B L_v^{(i)} \bar{X}^{(i)} (\hat{W}^{(i)})^\top \hat{E} + E V^{(i)} \bar{X}^{(i)} (L_w^{(i)})^\top \hat{C} - E V^{(i)} \bar{X}^{(i)} (L_w^{(i)})^\top L_v^{(i)} \bar{X}^{(i)} (\hat{W}^{(i)})^\top \hat{E}.
\]

Now consider
\begin{align}
& A V^{(i)} \bar{X}^{(i)} (\hat{W}^{(i)})^\top \hat{E} + E V^{(i)} \bar{X}^{(i)} (\hat{W}^{(i)})^\top \hat{A} - E V^{(i)} \bar{X}^{(i)} (\hat{W}^{(i)})^\top \hat{B} C V^{(i)} \bar{X}^{(i)} (\hat{W}^{(i)})^\top \hat{E} \nonumber \\
&\quad + B_{\perp}^{(i)}\hat{C}_{\perp}^{(i)} + B L_v^{(i)} \bar{X}^{(i)} (\hat{W}^{(i)})^\top \hat{E} + E V^{(i)} \bar{X}^{(i)} (L_w^{(i)})^\top \hat{C} - E V^{(i)} \bar{X}^{(i)} (L_w^{(i)})^\top L_v^{(i)} \bar{X}^{(i)} (\hat{W}^{(i)})^\top \hat{E} \nonumber \\
&= \big(A V^{(i)} + B L_v^{(i)}\big) \bar{X}^{(i)} (\hat{W}^{(i)})^\top \hat{E} + E V^{(i)} \bar{X}^{(i)} \big((\hat{W}^{(i)})^\top \hat{A} + (L_w^{(i)})^\top \hat{C}\big) + B_{\perp}^{(i)}\hat{C}_{\perp}^{(i)} \nonumber \\
&\quad + E V^{(i)} \bar{X}^{(i)} \big(-(\hat{W}^{(i)})^\top \hat{B} C V^{(i)} + (L_w^{(i)})^\top L_v^{(i)}\big) \bar{X}^{(i)} (\hat{W}^{(i)})^\top \hat{E} \nonumber \\
&= B_{\perp}^{(i)}\hat{C}_{\perp}^{(i)}.\nonumber
\end{align}
Using the definitions of \(R_x^{(i)}\) and the identities above, one can verify that \((W^{(i)})^\top R_x^{(i)} \hat{V}^{(i)} = 0\) and that \(R_x^{(i)} = B_{\perp}^{(i)} \hat{C}_{\perp}^{(i)}\).

\textbf{3. Recursive gain approximations:}\\
Note that \(\tilde{K}^{(i)} = E V^{(i)} \bar{X}^{(i)} (\hat{W}^{(i)})^\top \hat{B}\) and \(\bar{K}^{(i)} = C V^{(i)} \bar{X}^{(i)} (\hat{W}^{(i)})^\top \hat{E}\). Given the block diagonal structure of \(\bar{X}^{(i)}\), we directly obtain
\[
\tilde{K}^{(i)} = \tilde{K}^{(i-1)} + E v_i x_i w_i^\top \hat{B}, \qquad
\bar{K}^{(i)} = \bar{K}^{(i-1)} + C v_i x_i w_i^\top \hat{E}.
\]

\textbf{4. Sylvester equations for \(V^{(i)}\) and \(\hat{W}^{(i)}\):}\\
Consider
\begin{align}
& A V^{(i)} - E V^{(i)} A_r^{(i)} + B_{\perp}^{(i)} L_v^{(i)} \nonumber \\
&= A V^{(i)} - E V^{(i)} \big(S_v^{(i)} - B_r^{(i)} L_v^{(i)}\big) + \big(B - E V^{(i)} B_r^{(i)}\big) L_v^{(i)} \nonumber \\
&= 0.\nonumber
\end{align}
This establishes \eqref{Vb_Sylv}. By duality, a similar argument shows that \(\hat{W}^{(i)}\) satisfies \eqref{Wc_Sylv}.

This completes the proof.
\end{proof}
\begin{remark}
If a stabilizing solution is not required, then one can drop the condition that the ADI shifts have negative real parts. As long as the sets $(\alpha_1,\cdots,\alpha_i)$ and $(-\beta_1,\cdots,-\beta_i)$ do not have any common elements, the solution to the Sylvester equation \eqref{Sylv_X} remains unique and the approximation $X \approx V^{(i)} \bar{X}^{(i)} (\hat{W}^{(i)})^\top$ remains valid.
\end{remark}
\subsection{Algorithm}
This subsection discusses the algorithmic aspects of the proposed low-rank solver for NAREs. The key benefit of low-rank ADI methods is that no matrix equations are solved explicitly. Instead, these methods involve shifted linear matrix solves and basic matrix operations to recursively compute and accumulate low-rank factors of the solution to the respective matrix equations. We now shift our focus to avoiding the explicit solution of the Sylvester equations presented in the previous subsection.

Define \(y_i\) and \(z_i\) as follows:
\[
y_i = \big(A - \tilde{K}^{(i-1)} C + \alpha_i E\big)^{-1} B_{\perp}^{(i-1)}, \qquad
z_i = \big(\hat{A}^\top - (\bar{K}^{(i-1)})^\top \hat{B}^\top + \beta_i \hat{E}^\top\big)^{-1} (\hat{C}_{\perp}^{(i-1)})^\top.
\]
The vectors \(v_i\) and \(w_i\) are then constructed according to the shift case:

Case I: \[v_i = y_i,\quad w_i = z_i.\]

Case II: 
\[
v_i = \begin{bmatrix} \mathrm{Re}(y_i) & \mathrm{Im}(y_i) \end{bmatrix}, \qquad
w_i = \begin{bmatrix} \mathrm{Re}(z_i) & \mathrm{Im}(z_i) \end{bmatrix}.
\]

Case III: 
\[
v_i = \begin{bmatrix} \mathrm{Re}(y_i) & \mathrm{Im}(y_i) \end{bmatrix}, \qquad
w_i = \begin{bmatrix} z_i & z_i' \end{bmatrix},
\]
where 
\[
z_i' = \big(\hat{A}^\top - (\bar{K}^{(i-1)})^\top \hat{B}^\top + \beta_{i+1} \hat{E}^\top\big)^{-1} \hat{E}^\top z_i.
\]

Case IV: 
\[
v_i = \begin{bmatrix} y_i & y_i' \end{bmatrix}, \qquad
w_i = \begin{bmatrix} \mathrm{Re}(z_i) & \mathrm{Im}(z_i) \end{bmatrix},
\]
where 
\[
y_i' = \big(A - \tilde{K}^{(i-1)} C + \alpha_{i+1} E\big)^{-1} E y_i.
\]

The matrix \(x_i\) can be computed analytically as follows.

Case I:
\begin{align}
x_i = -(\alpha_i + \beta_i)\big(I_m + w_i^\top \hat{B} C v_i\big)^{-1}. \label{x_case1}
\end{align}

Case II:
\begin{align}
x_i = \Delta_d \begin{bmatrix} q_{11} & q_{12} \\ q_{21} & q_{22} \end{bmatrix}^{-1}, \label{x_case2_f}
\end{align}
where
\begin{align}
q_{11} &= \eta_{1,r}\delta_r + \eta_{1,i}\delta_i, \quad
q_{12} = \eta_{2,r}\delta_r + \eta_{2,i}\delta_i, \quad
q_{21} = \eta_{1,i}\delta_r - \eta_{1,r}\delta_i, \quad
q_{22} = \eta_{2,i}\delta_r - \eta_{2,r}\delta_i, \\[4pt]
\eta_{1,r} &= -\sigma_i I_m - \sigma_i\big( \mathrm{Re}(z_i)^\top \hat{B} - \mathrm{Im}(\beta_i)\,\mathrm{Im}(z_i)^\top \hat{B}\big)C\,\mathrm{Re}(y_i) - \mathrm{Im}(\alpha_i)\,\mathrm{Re}(z_i)^\top \hat{B} C\,\mathrm{Re}(y_i), \\[4pt]
\eta_{1,i} &= -\mathrm{Im}(\beta_i)I_m - \big(\sigma_i\,\mathrm{Im}(z_i)^\top \hat{B} + \mathrm{Im}(\beta_i)\,\mathrm{Re}(z_i)^\top \hat{B}\big)C\,\mathrm{Re}(y_i) - \mathrm{Im}(\alpha_i)\,\mathrm{Im}(z_i)^\top \hat{B} C\,\mathrm{Im}(y_i), \\[4pt]
\eta_{2,r} &= \mathrm{Im}(\alpha_i)I_m + \mathrm{Im}(\alpha_i)\,\mathrm{Re}(z_i)^\top \hat{B} C\,\mathrm{Re}(y_i) - \big(\sigma_i\,\mathrm{Re}(z_i)^\top \hat{B} - \mathrm{Im}(\beta_i)\,\mathrm{Im}(z_i)^\top \hat{B}\big)C\,\mathrm{Im}(y_i), \\[4pt]
\eta_{2,i} &= \mathrm{Im}(\alpha_i)\,\mathrm{Im}(z_i)^\top \hat{B} C\,\mathrm{Re}(y_i) - \big(\sigma_i\,\mathrm{Im}(z_i)^\top \hat{B} + \mathrm{Im}(\beta_i)\,\mathrm{Re}(z_i)^\top \hat{B}\big)C\,\mathrm{Im}(y_i), \\[4pt]
\sigma_i &= \mathrm{Re}(\alpha_i + \beta_i), \quad
\delta_r = \sigma_i^2 + \mathrm{Im}(\alpha_i)^2 - \mathrm{Im}(\beta_i)^2, \quad
\delta_i = 2\sigma_i\,\mathrm{Im}(\beta_i), \quad
\Delta_d = \delta_r^2 + \delta_i^2. \label{x_case2_l}
\end{align}

Case III:
\begin{align}
x_i = \begin{bmatrix} q_{11} & q_{12} \\ q_{21} & q_{22} \end{bmatrix}^{-1}, \label{x_case3_f}
\end{align}
where
\begin{align}
q_{11} &= -\frac{\sigma_a}{\delta_a}I_m - \frac{\sigma_a}{\delta_a}\big(z_i^\top \hat{B} C \,\mathrm{Re}(y_i)\big) - \frac{\mathrm{Im}(\alpha_i)}{\delta_a}\big(z_i^\top \hat{B} C \,\mathrm{Im}(y_i)\big), \\[4pt]
q_{12} &= \frac{\mathrm{Im}(\alpha_i)}{\delta_a}I_m + \frac{\mathrm{Im}(\alpha_i)}{\delta_a}\big(z_i^\top \hat{B} C \,\mathrm{Re}(y_i)\big) - \frac{\sigma_a}{\delta_a}\big(z_i^\top \hat{B} C \,\mathrm{Im}(y_i)\big), \\[4pt]
q_{21} &= \frac{\mathrm{Im}(\alpha_i)^2 - \sigma_a\mu_a}{\delta_a\epsilon_a}I_m 
+ \frac{\mathrm{Im}(\alpha_i)^2 - \sigma_a\mu_a}{\delta_a\epsilon_a}\big(z_i^\top \hat{B} C \,\mathrm{Re}(y_i)\big) 
- \frac{\mathrm{Im}(\alpha_i)(\sigma_a+\mu_a)}{\delta_a\epsilon_a}\big(z_i^\top \hat{B} C \,\mathrm{Im}(y_i)\big) \nonumber \\
&\quad - \frac{\mu_a}{\epsilon_a}\big((z_i')^\top \hat{B} C \,\mathrm{Re}(y_i)\big) - \frac{\mathrm{Im}(\alpha_i)}{\epsilon_a}\big((z_i')^\top \hat{B} C \,\mathrm{Im}(y_i)\big), \\[4pt]
q_{22} &= \frac{\mathrm{Im}(\alpha_i)(\sigma_a+\mu_a)}{\delta_a\epsilon_a}I_m 
+ \frac{\mathrm{Im}(\alpha_i)(\sigma_a+\mu_a)}{\delta_a\epsilon_a}\big(z_i^\top \hat{B} C \,\mathrm{Re}(y_i)\big) 
+ \frac{\mathrm{Im}(\alpha_i)^2 - \sigma_a\mu_a}{\delta_a\epsilon_a}\big(z_i^\top \hat{B} C \,\mathrm{Im}(y_i)\big) \nonumber \\
&\quad + \frac{\mathrm{Im}(\alpha_i)}{\epsilon_a}\big((z_i')^\top \hat{B} C \,\mathrm{Re}(y_i)\big) - \frac{\mu_a}{\epsilon_a}\big((z_i')^\top \hat{B} C \,\mathrm{Im}(y_i)\big), \\[4pt]
\sigma_a &= \mathrm{Re}(\alpha_i) + \beta_i, \quad
\mu_a = \mathrm{Re}(\alpha_i) + \beta_{i+1}, \quad
\delta_a = \sigma_a^2 + \mathrm{Im}(\alpha_i)^2, \quad
\epsilon_a = \mu_a^2 + \mathrm{Im}(\alpha_i)^2. \label{x_case3_l}
\end{align}

Case IV:
\begin{align}
x_i = \begin{bmatrix} q_{11} & q_{12} \\ q_{21} & q_{22} \end{bmatrix}^{-1}, \label{x_case4_f}
\end{align}
where
\begin{align}
q_{11} &= -\frac{\sigma_b}{\mu_b}I_m - \frac{\sigma_b}{\delta_b}\big(\mathrm{Re}(z_i)^\top \hat{B} C y_i\big) - \frac{\mathrm{Im}(\beta_i)}{\delta_b}\big(\mathrm{Im}(z_i)^\top \hat{B} C y_i\big), \\[4pt]
q_{12} &= \frac{\mathrm{Im}(\beta_i)^2 - \sigma_b\mu_b}{\delta_b\epsilon_b}I_m 
+ \frac{\mathrm{Im}(\beta_i)^2 - \sigma_b\mu_b}{\delta_b\epsilon_b}\big(\mathrm{Re}(z_i)^\top \hat{B} C y_i\big) 
- \frac{\mathrm{Im}(\beta_i)(\sigma_b+\mu_b)}{\delta_b\epsilon_b}\big(\mathrm{Im}(z_i)^\top \hat{B} C y_i\big) \nonumber \\
&\quad - \frac{\mu_b}{\epsilon_b}\big(\mathrm{Re}(z_i)^\top \hat{B} C y_i'\big) - \frac{\mathrm{Im}(\beta_i)}{\epsilon_b}\big(\mathrm{Im}(z_i)^\top \hat{B} C y_i'\big), \\[4pt]
q_{21} &= \frac{\mathrm{Im}(\beta_i)}{\delta_b}I_m + \frac{\mathrm{Im}(\beta_i)}{\delta_b}\big(\mathrm{Re}(z_i)^\top \hat{B} C y_i\big) - \frac{\sigma_b}{\delta_b}\big(\mathrm{Im}(z_i)^\top \hat{B} C y_i\big), \\[4pt]
q_{22} &= \frac{\mathrm{Im}(\beta_i)(\sigma_b+\mu_b)}{\delta_b\epsilon_b}I_m 
+ \frac{\mathrm{Im}(\beta_i)(\sigma_b+\mu_b)}{\delta_b\epsilon_b}\big(\mathrm{Re}(z_i)^\top \hat{B} C y_i\big) 
+ \frac{\mathrm{Im}(\beta_i)^2 - \sigma_b\mu_b}{\delta_b\epsilon_b}\big(\mathrm{Im}(z_i)^\top \hat{B} C y_i\big) \nonumber \\
&\quad + \frac{\mathrm{Im}(\beta_i)}{\epsilon_b}\big(\mathrm{Re}(z_i)^\top \hat{B} C y_i'\big) - \frac{\mu_b}{\epsilon_b}\big(\mathrm{Im}(z_i)^\top \hat{B} C y_i'\big), \\[4pt]
\sigma_b &= \alpha_i + \mathrm{Re}(\beta_i), \quad
\mu_b = \alpha_{i+1} + \mathrm{Re}(\beta_i), \quad
\delta_b = \sigma_b^2 + \mathrm{Im}(\beta_i)^2, \quad
\epsilon_b = \mu_b^2 + \mathrm{Im}(\beta_i)^2. \label{x_case4_l}
\end{align}

The pseudo-code of the low-rank ADI algorithm for NAREs (N-RADI) is given in Algorithm 1.\\
\\
\\
\\
\\
\\
\\
\noindent\rule{\textwidth}{0.8pt}

\textbf{Algorithm 1: N-RADI}

\noindent\rule{\textwidth}{0.8pt}

\textbf{Input:} Matrices of NARE \eqref{nare}: $A$, $B$, $C$, $E$, $\hat{A}$, $\hat{A}$, $\hat{B}$, $\hat{C}$, $\hat{E}$; ADI shifts: $\{\alpha_i\}_{i=1}^{k}\in\mathbb{C}_{-}$, $\{\beta_i\}_{i=1}^{k}\in\mathbb{C}_{-}$; Tolerance: $\tau\in[0,1]$. 

\textbf{Output:} Approximation of $X$: $X\approx\tilde{X}^{(i)}=V^{(i)}\bar{X}^{(i)}(\hat{W}^{(i)})^\top$; Approximations of gain matrices: $K_{\mathrm{gain}}\approx \tilde{K}^{(i)}=E\tilde{X}^{(i)}\hat{B}$, $\hat{K}_{\mathrm{gain}}\approx\bar{K}^{(i)}= C\tilde{X}^{(i)}\hat{E}$; Residual: $R_x^{(i)}=B_{\perp}^{(i)}\hat{C}_{\perp}^{(i)}$.

1. \textbf{Initialization:} $i=0$, $B_{\perp}^{(i)}=B$, $\hat{C}_{\perp}^{(i)}=\hat{C}$, $\tilde{K}^{(i)}=0$, $\bar{K}^{(i)}=0$, $V^{(i)}=[\;]$, $\bar{X}^{(i)}=[\;]$, $\hat{W}^{(i)}=[\;]$.

2. \textbf{while} $\frac{\|B_{\perp}^{(i)}\hat{C}_{\perp}^{(i)}\|}{\|B\hat{C}\|}\geq \tau$ \textbf{do}

3. Solve for $y_i$ and $z_i$:

\hspace*{2cm}$\big(A-\tilde{K}^{(i-1)}C+\alpha_i E\big)y_i=B_{\perp}^{(i-1)}$,\quad $\big(\hat{A}^\top-(\bar{K}^{(i-1)})^\top \hat{B}^\top +\beta_i \hat{E}^\top\big)z_i=(\hat{C}_{\perp}^{(i-1)})^\top$.

4. \textbf{If} $\mathrm{Im}(\alpha_i)=0$ $\land$ $\mathrm{Im}(\beta_i)=0$ \textbf{then}

5. Set $v_i=y_i$ and $w_i=z_i$, and compute $x_i$ from \eqref{x_case1}.

6. Expand $V^{(i)}=\begin{bmatrix}V^{(i-1)}&v_i\end{bmatrix}$, $\hat{W}^{(i)}=\begin{bmatrix}\hat{W}^{(i-1)}&w_i\end{bmatrix}$, and $\bar{X}^{(i)}=\mathrm{blkdiag}(\bar{X}^{(i-1)},x_i)$.

7. Update $B_{\perp}^{(i)}=B_{\perp}^{(i-1)}+Ev_ix_i$, $\hat{C}_{\perp}^{(i)}=\hat{C}_{\perp}^{(i-1)}+x_iw_i^\top\hat{E}$, $\tilde{K}^{(i)}=\tilde{K}^{(i-1)}+Ev_ix_iw_i^\top\hat{B}$, and

\hspace*{0.5cm}$\bar{K}^{(i)}=\bar{K}^{(i-1)}+Cv_ix_iw_i^\top\hat{E}$, $i=i+1$.

8. \textbf{else if} $\mathrm{Im}(\alpha_i)\neq0$ $\land$ $\mathrm{Im}(\beta_i)\neq0$ \textbf{then}

9. Set $v_i=\begin{bmatrix}\mathrm{Re}(y_i)&\mathrm{Im}(y_i)\end{bmatrix}$ and $w_i=\begin{bmatrix}\mathrm{Re}(z_i)&\mathrm{Im}(z_i)\end{bmatrix}$, and compute $x_i$ from \eqref{x_case2_f}-\eqref{x_case2_l}.

10. Expand $V^{(i)}=\begin{bmatrix}V^{(i-1)}&v_i\end{bmatrix}$, $\hat{W}^{(i)}=\begin{bmatrix}\hat{W}^{(i-1)}&w_i\end{bmatrix}$, and $\bar{X}^{(i)}=\mathrm{blkdiag}(\bar{X}^{(i-1)},x_i)$.

11. Update $B_{\perp}^{(i)}=B_{\perp}^{(i-1)}+Ev_ix_i\begin{bmatrix}I_m\\0\end{bmatrix}$, $\hat{C}_{\perp}^{(i)}=\hat{C}_{\perp}^{(i-1)}+\begin{bmatrix}I_m&0\end{bmatrix}x_iw_i^\top\hat{E}$, $\tilde{K}^{(i)}=\tilde{K}^{(i-1)}+Ev_ix_iw_i^\top\hat{B}$,

\hspace*{0.5cm}and $\bar{K}^{(i)}=\bar{K}^{(i-1)}+Cv_ix_iw_i^\top\hat{E}$, $i=i+2$.

12. \textbf{else if} $\mathrm{Im}(\alpha_i)\neq0$ $\land$ $\mathrm{Im}(\beta_i)=0$ $\land$ $\mathrm{Im}(\beta_{i+1})=0$ \textbf{then}

13. Solve for $z^{\prime}_i:$ $\big(\hat{A}^\top-(\bar{K}^{(i-1)})^\top \hat{B}^\top +\beta_{i+1} \hat{E}^\top\big)z^{\prime}_i=\hat{E}^\top z_i$.

14. Set $v_i=\begin{bmatrix}\mathrm{Re}(y_i)&\mathrm{Im}(y_i)\end{bmatrix}$ and $w_i=\begin{bmatrix}z_i&z_i^{\prime}\end{bmatrix}$, and compute $x_i$ from \eqref{x_case3_f}-\eqref{x_case3_l}.

15. Expand $V^{(i)}=\begin{bmatrix}V^{(i-1)}&v_i\end{bmatrix}$, $\hat{W}^{(i)}=\begin{bmatrix}\hat{W}^{(i-1)}&w_i\end{bmatrix}$, and $\bar{X}^{(i)}=\mathrm{blkdiag}(\bar{X}^{(i-1)},x_i)$.

16. Update $B_{\perp}^{(i)}=B_{\perp}^{(i-1)}+Ev_ix_i\begin{bmatrix}I_m\\0\end{bmatrix}$, $\hat{C}_{\perp}^{(i)}=\hat{C}_{\perp}^{(i-1)}+\begin{bmatrix}I_m&0\end{bmatrix}x_iw_i^\top\hat{E}$, $\tilde{K}^{(i)}=\tilde{K}^{(i-1)}+Ev_ix_iw_i^\top\hat{B}$,

\hspace*{0.5cm}and $\bar{K}^{(i)}=\bar{K}^{(i-1)}+Cv_ix_iw_i^\top\hat{E}$, $i=i+2$.

17. \textbf{else if} $\mathrm{Im}(\alpha_i)=0$ $\land$ $\mathrm{Im}(\alpha_{i+1})=0$ $\land$ $\mathrm{Im}(\beta_{i})\neq0$ \textbf{then}

18. Solve for $y^{\prime}_i:$ $\big(A-\tilde{K}^{(i-1)}C+\alpha_{i+1} E\big)y^{\prime}_i=Ey_i$.

19. Set $v_i=\begin{bmatrix}y_i&y_i^{\prime}\end{bmatrix}$ and $w_i=\begin{bmatrix}\mathrm{Re}(z_i)&\mathrm{Im}(z_i)\end{bmatrix}$, and compute $x_i$ from \eqref{x_case4_f}-\eqref{x_case4_l}.

20. Expand $V^{(i)}=\begin{bmatrix}V^{(i-1)}&v_i\end{bmatrix}$, $\hat{W}^{(i)}=\begin{bmatrix}\hat{W}^{(i-1)}&w_i\end{bmatrix}$, and $\bar{X}^{(i)}=\mathrm{blkdiag}(\bar{X}^{(i-1)},x_i)$.

21. Update $B_{\perp}^{(i)}=B_{\perp}^{(i-1)}+Ev_ix_i\begin{bmatrix}I_m\\0\end{bmatrix}$, $\hat{C}_{\perp}^{(i)}=\hat{C}_{\perp}^{(i-1)}+\begin{bmatrix}I_m&0\end{bmatrix}x_iw_i^\top\hat{E}$, $\tilde{K}^{(i)}=\tilde{K}^{(i-1)}+Ev_ix_iw_i^\top\hat{B}$,

\hspace*{0.5cm}and $\bar{K}^{(i)}=\bar{K}^{(i-1)}+Cv_ix_iw_i^\top\hat{E}$, $i=i+2$.

22. \textbf{end if}

23. \textbf{end while}\\

\noindent\rule{\textwidth}{0.8pt}

The shifted linear solves to compute \(y_i\), \(y_i^\prime\), \(z_i\), and \(z_i^\prime\) require efficient implementation, which is achieved using the SMW formula. Thus the linear solves should be rewritten as
\[
\big(A + \alpha E\big) y = \begin{bmatrix} B_{\perp} & \tilde{K} \end{bmatrix}, \qquad
\big(\hat{A}^\top + \beta \hat{E}\big) z = \begin{bmatrix} \hat{C}_{\perp}^\top & \bar{K}^\top \end{bmatrix}.
\]
Depending on how large \(p\) is, the computational cost of these linear solves can be higher than that of the linear solves in low-rank ADI methods for Lyapunov and Sylvester equations \cite{benner2013reformulated,benner2014computing}. This issue is addressed in the next subsection.
\subsection{Avoiding SMW Formulas}
Let us define $s_{v,\mathrm{lyap}}^{(i)}$, $l_{v,\mathrm{lyap}}^{(i)}$, $s_{w,\mathrm{lyap}}^{(i)}$, and $l_{w,\mathrm{lyap}}^{(i)}$ as follows:
\begin{align}
s_{v,\mathrm{lyap}}^{(i)} & = 
\begin{cases} 
-\alpha_i I_{m}, \hspace*{5.9cm} \text{if } \mathrm{Im}(\alpha_i) = 0, \\[6pt]
\gamma_{v,i}^2\begin{bmatrix} 
I_{m} & \frac{\sqrt{1+\delta_{v,i}^2}}{2\delta_{v,i}}I_{m} \\ 
&\\
-\frac{\sqrt{1+\delta_{v,i}^2}}{2\delta_{v,i}}I_{m} & 0
\end{bmatrix}, \hspace*{2cm} \text{if } \mathrm{Im}(\alpha_i) \neq 0,
\end{cases}\label{sz}\\[6pt]
l_{v,\mathrm{lyap}}^{(i)} &= 
\begin{cases} 
-\gamma_{v,i} I_{m}, \hspace*{5.75cm} \text{if } \mathrm{Im}(\alpha_i) = 0, \\[6pt]
-\sqrt{2}\gamma_{v,i}\begin{bmatrix} 
I_{m} & 0
\end{bmatrix}, \hspace*{4.3cm} \text{if } \mathrm{Im}(\alpha_i) \neq 0,
\end{cases}\\[6pt]
s_{w,\mathrm{lyap}}^{(i)} & = 
\begin{cases} 
-\beta_i I_{m}, \hspace*{5.95cm} \text{if } \mathrm{Im}(\beta_i) = 0, \\[6pt]
\gamma_{w,i}^2\begin{bmatrix} 
I_{m} & \frac{\sqrt{1+\delta_{w,i}^2}}{2\delta_{w,i}}I_{m} \\ 
&\\
-\frac{\sqrt{1+\delta_{w,i}^2}}{2\delta_{w,i}}I_{m} & 0
\end{bmatrix}, \hspace*{1.85cm} \text{if } \mathrm{Im}(\beta_i) \neq 0,
\end{cases}\\[6pt]
l_{w,\mathrm{lyap}}^{(i)} &= 
\begin{cases} 
-\gamma_{w,i} I_{m}, \hspace*{5.7cm} \text{if } \mathrm{Im}(\beta_i) = 0, \\[6pt]
-\sqrt{2}\gamma_{w,i}\begin{bmatrix} 
I_{m} & 0
\end{bmatrix}, \hspace*{4.25cm} \text{if } \mathrm{Im}(\beta_i) \neq 0,
\end{cases}\label{lz}
\end{align}
where $\gamma_{v,i}=\sqrt{-2\mathrm{Re}(\alpha_i)}$, $\delta_{v,i}=\frac{\mathrm{Re}(\alpha_i)}{\mathrm{Im}(\alpha_i)}$, $\gamma_{w,i}=\sqrt{-2\mathrm{Re}(\beta_i)}$, and $\delta_{w,i}=\frac{\mathrm{Re}(\beta_i)}{\mathrm{Im}(\beta_i)}$.

Now define \(S_{v,\mathrm{lyap}}^{(i)}\), \(L_{v,\mathrm{lyap}}^{(i)}\), \(V_{\mathrm{lyap}}^{(i)}\), \(S_{w,\mathrm{lyap}}^{(i)}\), \(L_{w,\mathrm{lyap}}^{(i)}\), and \(\hat{W}_{\mathrm{lyap}}^{(i)}\) as follows:
\begin{align}
S_{v,\mathrm{lyap}}^{(i)}&=\begin{bmatrix}S_{v,\mathrm{lyap}}^{(i-1)}&(L_{v,\mathrm{lyap}}^{(i-1)})^\top l_{v,\mathrm{lyap}}^{(i)}\\0& s_{v,\mathrm{lyap}}^{(i)}\end{bmatrix},& L_{v,\mathrm{lyap}}^{(i)}&=\begin{bmatrix}L_{v,\mathrm{lyap}}^{(i-1)}&l_{v,\mathrm{lyap}}^{(i)}\end{bmatrix},& V_{\mathrm{lyap}}^{(i)}&=\begin{bmatrix}V_{\mathrm{lyap}}^{(i-1)}&v_{\mathrm{lyap}}^{(i)}\end{bmatrix},\nonumber\\[6pt]
S_{w,\mathrm{lyap}}^{(i)}&=\begin{bmatrix}S_{w,\mathrm{lyap}}^{(i-1)}&(L_{w,\mathrm{lyap}}^{(i-1)})^\top l_{w,\mathrm{lyap}}^{(i)}\\0& s_{w,\mathrm{lyap}}^{(i)}\end{bmatrix},& L_{w,\mathrm{lyap}}^{(i)}&=\begin{bmatrix}L_{w,\mathrm{lyap}}^{(i-1)}&l_{w,\mathrm{lyap}}^{(i)}\end{bmatrix},& \hat{W}_{\mathrm{lyap}}^{(i)}&=\begin{bmatrix}\hat{W}_{\mathrm{lyap}}^{(i-1)}&\hat{w}_{\mathrm{lyap}}^{(i)}\end{bmatrix}.\nonumber
\end{align}
Moreover, assume that \(V_{\mathrm{lyap}}^{(i)}\) and \(\hat{W}_{\mathrm{lyap}}^{(i)}\) satisfy the Sylvester equations
\begin{align}
A V_{\mathrm{lyap}}^{(i)}-E V_{\mathrm{lyap}}^{(i)} S_{v,\mathrm{lyap}}^{(i)}+B L_{v,\mathrm{lyap}}^{(i)}=0,\\
\hat{A}^\top \hat{W}_{\mathrm{lyap}}^{(i)}-\hat{E}^\top \hat{W}_{\mathrm{lyap}}^{(i)} S_{w,\mathrm{lyap}}^{(i)}+\hat{C}^\top L_{w,\mathrm{lyap}}^{(i)}=0.
\end{align}
As shown in \cite{wolfthesis,zulfiqar2025unified}, the low-rank approximate solutions of the Lyapunov equations
\begin{align}
A P_{\mathrm{lyap}} E^\top+ E P_{\mathrm{lyap}} A^\top+ BB^\top=0,\label{lyap_p}\\
\hat{A}^\top \hat{Q}_{\mathrm{lyap}} \hat{E}+ \hat{E}^\top \hat{Q}_{\mathrm{lyap}} \hat{A}+ \hat{C}^\top \hat{C}=0,\label{lyap_q}
\end{align}
obtained via the CF-ADI method \cite{benner2013reformulated} are given by \(P_{\mathrm{lyap}}\approx V_{\mathrm{lyap}}^{(i)}(V_{\mathrm{lyap}}^{(i)})^\top\) and \(\hat{Q}_{\mathrm{lyap}}\approx \hat{W}_{\mathrm{lyap}}^{(i)}(\hat{W}_{\mathrm{lyap}}^{(i)})^\top\).

Let us define \( y_{\mathrm{lyap}}^{(i)} \) and \( z_{\mathrm{lyap}}^{(i)} \) as
\begin{align}
y_{\mathrm{lyap}}^{(i)}&=\big(A+\alpha_i E\big)^{-1}\mathcal{B}_{\perp}^{(i-1)},\\
z_{\mathrm{lyap}}^{(i)}&=\big(\hat{A}^\top+\beta_i \hat{E}^\top\big)^{-1}\big(\hat{\mathcal{C}}_{\perp}^{(i-1)}\big)^\top,
\end{align}
where \( \mathcal{B}_{\perp}^{(i)} \) and \( \hat{\mathcal{C}}_{\perp}^{(i)} \) are updated recursively as
\begin{align}
\mathcal{B}_{\perp}^{(i)}&=\begin{cases} 
\mathcal{B}_{\perp}^{(i-1)}+\gamma_{v,i}^2Ey_{\mathrm{lyap}}^{(i)}, \hspace*{5cm} \text{if } \mathrm{Im}(\alpha_i) = 0, \\
\mathcal{B}_{\perp}^{(i-1)}+2\gamma_{v,i}^2E\Big(\mathrm{Re}\big(y_{\mathrm{lyap}}^{(i)}\big)+\delta_{v,i}\mathrm{Im}\big(y_{\mathrm{lyap}}^{(i)}\big)\Big), \hspace*{1.3cm} \text{if } \mathrm{Im}(\alpha_i) \neq 0,
\end{cases}\\[6pt]
\hat{\mathcal{C}}_{\perp}^{(i)} &= 
\begin{cases} 
\hat{\mathcal{C}}_{\perp}^{(i-1)}+\gamma_{w,i}^2(z_{\mathrm{lyap}}^{(i)})^\top \hat{E}, \hspace*{4.5cm} \text{if } \mathrm{Im}(\beta_i) = 0, \\
\hat{\mathcal{C}}_{\perp}^{(i-1)}+2\gamma_{w,i}^2\Big(\mathrm{Re}\big(z_{\mathrm{lyap}}^{(i)}\big)+\delta_{w,i}\mathrm{Im}\big(z_{\mathrm{lyap}}^{(i)}\big)\Big)^\top \hat{E}, \hspace*{1cm} \text{if } \mathrm{Im}(\beta_i) \neq 0,
\end{cases}\label{C_}
\end{align}
with \( \mathcal{B}_{\perp}^{(0)} = B \) and \( \hat{\mathcal{C}}_{\perp}^{(0)} = \hat{C} \).

As shown in \cite{benner2013reformulated}, \(v_{\mathrm{lyap}}^{(i)}\) and \(\hat{w}_{\mathrm{lyap}}^{(i)}\) are computed in CF-ADI \cite{benner2013reformulated} to approximate \(P_{\mathrm{lyap}}\) and \(\hat{Q}_{\mathrm{lyap}}\) as follows:
\begin{align}
v_{\mathrm{lyap}}^{(i)} &= 
\begin{cases} 
\gamma_{v,i}y_{\mathrm{lyap}}^{(i)}, \hspace*{10cm} \text{if } \mathrm{Im}(\alpha_i) = 0, \\
 \begin{bmatrix}\sqrt{2}\gamma_{v,i}\Big(\mathrm{Re}\big(y_{\mathrm{lyap}}^{(i)}\big)+\delta_{v,i}\mathrm{Im}\big(y_{\mathrm{lyap}}^{(i)}\big)\Big)& \sqrt{2}\gamma_{v,i}\sqrt{\delta_{v,i}^2+1}\mathrm{Im}\big(y_{\mathrm{lyap}}^{(i)}\big)\end{bmatrix}, \hspace*{1.35cm} \text{if } \mathrm{Im}(\alpha_i) \neq 0,
\end{cases}\\[6pt]
\hat{w}_{\mathrm{lyap}}^{(i)} &= 
\begin{cases} 
\gamma_{w,i}z_{\mathrm{lyap}}^{(i)}, \hspace*{9.9cm} \text{if } \mathrm{Im}(\beta_i) = 0, \\
 \begin{bmatrix}\sqrt{2}\gamma_{w,i}\Big(\mathrm{Re}\big(z_{\mathrm{lyap}}^{(i)}\big)+\delta_{w,i}\mathrm{Im}\big(z_{\mathrm{lyap}}^{(i)}\big)\Big)& \sqrt{2}\gamma_{w,i}\sqrt{\delta_{w,i}^2+1}\mathrm{Im}\big(z_{\mathrm{lyap}}^{(i)}\big)\end{bmatrix}, \hspace*{1.1cm} \text{if } \mathrm{Im}(\beta_i) \neq 0.
\end{cases}
\end{align}

The matrices \(V_{\mathrm{lyap}}^{(i)}\) and \(\hat{W}_{\mathrm{lyap}}^{(i)}\) also satisfy the following Sylvester equations:
\begin{align}
A V_{\mathrm{lyap}}^{(i)}+E V_{\mathrm{lyap}}^{(i)} (S_{v,\mathrm{lyap}}^{(i)})^\top+\mathcal{B}_{\perp}^{(i)} L_{v,\mathrm{lyap}}^{(i)}&=0,\\
\hat{A}^\top \hat{W}_{\mathrm{lyap}}^{(i)}+\hat{E}^\top \hat{W}_{\mathrm{lyap}}^{(i)} (S_{w,\mathrm{lyap}}^{(i)})^\top+(\hat{\mathcal{C}}_{\perp}^{(i)})^\top L_{w,\mathrm{lyap}}^{(i)}&=0;\label{W_lyap_sylv2}
\end{align}
cf. \cite{wolfthesis,zulfiqar2025unified}.

It is shown in \cite{wolf2016adi} that \(V_{\mathrm{lyap}}^{(i)}\) and \(\hat{W}_{\mathrm{lyap}}^{(i)}\) satisfy the properties \eqref{int_prop1} and \eqref{int_prop2}. Consequently, there exist invertible matrices \(T_{v}^{(i)}\) and \(T_{w}^{(i)}\) such that \(V^{(i)} = V_{\mathrm{lyap}}^{(i)} T_{v}^{(i)}\) and \(\hat{W}^{(i)} = \hat{W}_{\mathrm{lyap}}^{(i)} T_{w}^{(i)}\). The recursive formulas for these matrices are given in the next proposition.
\begin{proposition}\label{th2}
Let us assume that all the variables in Theorem \ref{th1} and \eqref{sz}-\eqref{W_lyap_sylv2} hold. Define \(T_{v}^{(i)}\), \(T_{w}^{(i)}\), \(\hat{B}_{r,\mathrm{lyap}}^{(i)}\), and \(C_{r,\mathrm{lyap}}^{(i)}\) as follows:
\begin{align}
T_{v}^{(i)}=\begin{bmatrix}T_{v}^{(i-1)}&t_{v,1}^{(i)}\\0&t_{v,2}^{(i)}\end{bmatrix},\quad T_{w}^{(i)}=\begin{bmatrix}T_{w}^{(i-1)}&t_{w,1}^{(i)}\\0&t_{w,2}^{(i)}\end{bmatrix},\quad \hat{B}_{r,\mathrm{lyap}}^{(i)}=(\hat{W}_{\mathrm{lyap}}^{(i)})^\top\hat{B},\quad C_{r,\mathrm{lyap}}^{(i)}=C V_{\mathrm{lyap}}^{(i)}.
\end{align}
Further, define \(t_{v,\mathrm{lyap}}^{(i)}\), \(t_{v,\mathrm{lyap}}^{\prime(i)}\), \(t_{w,\mathrm{lyap}}^{(i)}\), and \(t_{w,\mathrm{lyap}}^{\prime(i)}\) as follows:
\begin{align}
t_{v,\mathrm{lyap}}^{(i)}&=\Big(-(S_{v,\mathrm{lyap}}^{(i)})^\top-\begin{bmatrix}T_{v}^{(i-1)}\\0\end{bmatrix}\bar{X}^{(i-1)}\hat{B}_r^{(i-1)}C_{r,\mathrm{lyap}}^{(i)}+\alpha_iI\Big)^{-1}
\Big((L_{v,\mathrm{lyap}}^{(i)})^\top-\begin{bmatrix}T_{v}^{(i-1)}\\0\end{bmatrix}B_r^{(i-1)}\Big),\label{tv}\\
t_{v,\mathrm{lyap}}^{\prime(i)}&=\Big(-(S_{v,\mathrm{lyap}}^{(i)})^\top-\begin{bmatrix}T_{v}^{(i-1)}\\0\end{bmatrix}\bar{X}^{(i-1)}\hat{B}_r^{(i-1)}C_{r,\mathrm{lyap}}^{(i)}+\alpha_{i+1}I\Big)^{-1}t_{v,\mathrm{lyap}}^{(i)},\label{tvp}\\
t_{w,\mathrm{lyap}}^{(i)}&=\Big(-(S_{w,\mathrm{lyap}}^{(i)})^\top-\begin{bmatrix}T_{w}^{(i-1)}\\0\end{bmatrix}(\bar{X}^{(i-1)})^\top (C_r^{(i-1)})^\top(\hat{B}_{r,\mathrm{lyap}}^{(i)})^\top+\beta_iI\Big)^{-1}
\Big((L_{w,\mathrm{lyap}}^{(i)})^\top\nonumber\\
&\qquad-\begin{bmatrix}T_{w}^{(i-1)}\\0\end{bmatrix}\big(\hat{C}_r^{(i-1)})^\top\Big),\label{tw}\\
t_{w,\mathrm{lyap}}^{\prime(i)}&=\Big(-(S_{w,\mathrm{lyap}}^{(i)})^\top-\begin{bmatrix}T_{w}^{(i-1)}\\0\end{bmatrix}(\bar{X}^{(i-1)})^\top (C_r^{(i-1)})^\top(\hat{B}_{r,\mathrm{lyap}}^{(i)})^\top+\beta_{i+1}I\Big)^{-1}t_{w,\mathrm{lyap}}^{(i)}.\label{twp}
\end{align}
Next, set \(t_{v}^{(i)}\) and \(t_{w}^{(i)}\) as follows:

For Case I:
\begin{align}
t_{v}^{(i)}=\begin{bmatrix}t_{v,1}^{(i)}\\t_{v,2}^{(i)}\end{bmatrix}=t_{v,\mathrm{lyap}}^{(i)}\quad \text{and}\quad  t_{w}^{(i)}=\begin{bmatrix}t_{w,1}^{(i)}\\t_{w,2}^{(i)}\end{bmatrix}=t_{w,\mathrm{lyap}}^{(i)}.
\end{align}

For Case II:
\begin{align}
t_{v}^{(i)}=\begin{bmatrix}t_{v,1}^{(i)}\\t_{v,2}^{(i)}\end{bmatrix}=\begin{bmatrix}\mathrm{Re}\big(t_{v,\mathrm{lyap}}^{(i)}\big)&\mathrm{Im}\big(t_{v,\mathrm{lyap}}^{(i)}\big)\end{bmatrix}\quad \text{and}\quad  t_{w}^{(i)}=\begin{bmatrix}t_{w,1}^{(i)}\\t_{w,2}^{(i)}\end{bmatrix}=\begin{bmatrix}\mathrm{Re}\big(t_{w,\mathrm{lyap}}^{(i)}\big)&\mathrm{Im}\big(t_{w,\mathrm{lyap}}^{(i)}\big)\end{bmatrix}.
\end{align}

For Case III:
\begin{align}
t_{v}^{(i)}=\begin{bmatrix}t_{v,1}^{(i)}\\t_{v,2}^{(i)}\end{bmatrix}=\begin{bmatrix}\mathrm{Re}\big(t_{v,\mathrm{lyap}}^{(i)}\big)&\mathrm{Im}\big(t_{v,\mathrm{lyap}}^{(i)}\big)\end{bmatrix}\quad \text{and}\quad  t_{w}^{(i)}=\begin{bmatrix}t_{w,1}^{(i)}\\t_{w,2}^{(i)}\end{bmatrix}=\begin{bmatrix}t_{w,\mathrm{lyap}}^{(i)}&t_{w,\mathrm{lyap}}^{\prime(i)}\end{bmatrix}.
\end{align}

For Case IV:
\begin{align}
t_{v}^{(i)}=\begin{bmatrix}t_{v,1}^{(i)}\\t_{v,2}^{(i)}\end{bmatrix}=\begin{bmatrix}t_{v,\mathrm{lyap}}^{(i)}&t_{v,\mathrm{lyap}}^{\prime(i)}\end{bmatrix}\quad \text{and}\quad  t_{w}^{(i)}=\begin{bmatrix}t_{w,1}^{(i)}\\t_{w,2}^{(i)}\end{bmatrix}=\begin{bmatrix}\mathrm{Re}\big(t_{w,\mathrm{lyap}}^{(i)}\big)&\mathrm{Im}\big(t_{w,\mathrm{lyap}}^{(i)}\big)\end{bmatrix}.
\end{align}

Then \(v_i\), \(V^{(i)}\), \(w_i\), and \(\hat{W}^{(i)}\) can be extracted from \(V_{\mathrm{lyap}}^{(i)}\) and \(\hat{W}_{\mathrm{lyap}}^{(i)}\) as follows:
\[
v_i = V_{\mathrm{lyap}}^{(i)} t_{v}^{(i)},\quad V^{(i)} = V_{\mathrm{lyap}}^{(i)} T_{v}^{(i)},\quad w_i = \hat{W}_{\mathrm{lyap}}^{(i)} t_{w}^{(i)},\quad \hat{W}^{(i)} = \hat{W}_{\mathrm{lyap}}^{(i)} T_{w}^{(i)}.
\]
\end{proposition}
\begin{proof}
The proof is similar to that of Theorem 2.3 of \cite{zulfiqar2026ldl}, and hence omitted for brevity.
\end{proof}

It is clear that Lyapunov equations \eqref{lyap_p} and \eqref{lyap_q} can be solved simultaneously with \eqref{nare}. In fact, several other matrix equations can be solved simultaneously with shared linear solves, as done in the unified ADI (UADI) framework; see \cite{zulfiqar2025unified}. The unified ADI algorithm for NAREs (UN-RADI) is given in Algorithm 2. It further enhances the applicability of UADI, which reuses the same shifted linear solves for solving several matrix equations simultaneously.\\

\noindent\rule{\textwidth}{0.8pt}

\textbf{Algorithm 2: UN-RADI}

\noindent\rule{\textwidth}{0.8pt}

\textbf{Input:} Matrices of NARE \eqref{nare}: $A$, $B$, $C$, $E$, $\hat{A}$, $\hat{A}$, $\hat{B}$, $\hat{C}$, $\hat{E}$; ADI shifts: $\{\alpha_i\}_{i=1}^{k}\in\mathbb{C}_{-}$, $\{\beta_i\}_{i=1}^{k}\in\mathbb{C}_{-}$; Tolerance: $\tau\in[0,1]$. \\

\textbf{Output:} Approximation of $X$: $X\approx\tilde{X}^{(i)}=V_{\mathrm{lyap}}^{(i)}T_{v}^{(i)}\bar{X}^{(i)}(T_{w}^{(i)})^\top(\hat{W}_{\mathrm{lyap}}^{(i)})^\top$; Approximations of gain matrices: $K_{\mathrm{gain}}\approx \tilde{K}^{(i)}=E\tilde{X}^{(i)}\hat{B}$, $\hat{K}_{\mathrm{gain}}\approx\bar{K}^{(i)}= C\tilde{X}^{(i)}\hat{E}$; Residual: $R_x^{(i)}=B_{\perp}^{(i)}\hat{C}_{\perp}^{(i)}$.\\

1. \textbf{Initialization:} $i=0$, $B_{\perp}^{(i)}=B$, $\mathcal{B}_{\perp}^{(i)}=B$, $\hat{C}_{\perp}^{(i)}=\hat{C}$, $\hat{\mathcal{C}}_{\perp}^{(i)}=\hat{C}$, $\tilde{K}^{(i)}=0$, $\bar{K}^{(i)}=0$, $V_{\mathrm{lyap}}^{(i)}=[\;]$, $T_{v}^{(i)}=[\;]$, $\bar{X}^{(i)}=[\;]$, $\hat{W}_{\mathrm{lyap}}^{(i)}=[\;]$, $T_{w}^{(i)}=[\;]$, $S_{v,\mathrm{lyap}}^{(i)}=[\;]$, $L_{v,\mathrm{lyap}}^{(i)}=[\;]$, $S_{w,\mathrm{lyap}}^{(i)}=[\;]$, $L_{w,\mathrm{lyap}}^{(i)}=[\;]$, $\hat{B}_{r,\mathrm{lyap}}^{(i)}$=[\;], $\hat{B}_r^{(i)}=[\;]$, $B_r^{(i)}=[\;]$, $C_{r,\mathrm{lyap}}^{(i)}$=[\;], $C_r^{(i)}=[\;]$, $\hat{C}_r^{(i)}=[\;]$.\\

2. \textbf{while} $\frac{\|B_{\perp}^{(i)}\hat{C}_{\perp}^{(i)}\|}{\|B\hat{C}\|}\geq \tau$ \textbf{do}\\

3. Solve for $y_{\mathrm{lyap}}^{(i)}$ and $z_{\mathrm{lyap}}^{(i)}$:

\[
\big(A+\alpha_i E\big)y_{\mathrm{lyap}}^{(i)}=\mathcal{B}_{\perp}^{(i-1)}\quad \text{and}\quad \big(\hat{A}^\top+\beta_i \hat{E}^\top\big)z_{\mathrm{lyap}}^{(i)}=\big(\hat{\mathcal{C}}_{\perp}^{(i-1)}\big)^\top.
\]

4. Set $\gamma_{v,i}=\sqrt{-2\mathrm{Re}(\alpha_i)}$ and $\gamma_{w,i}=\sqrt{-2\mathrm{Re}(\beta_i)}$.\\

5. \textbf{If} $\mathrm{Im}(\alpha_i)=0$ $\land$ $\mathrm{Im}(\beta_i)=0$ \textbf{then}\\

6. Set $v_{\mathrm{lyap}}^{(i)}=\gamma_{v,i}y_{\mathrm{lyap}}^{(i)}$, $s_{v,\mathrm{lyap}}^{(i)}=-\alpha_i I_m$, $l_{v,\mathrm{lyap}}^{(i)}=-\gamma_{v,i} I_m$, $w_{\mathrm{lyap}}^{(i)}=\gamma_{w,i}z_{\mathrm{lyap}}^{(i)}$, $s_{w,\mathrm{lyap}}^{(i)}=-\beta_i I_m$, and $l_{w,\mathrm{lyap}}^{(i)}=-\gamma_{w,i} I_m$.

7. Update $\mathcal{B}_{\perp}^{(i)}=\mathcal{B}_{\perp}^{(i-1)}+\gamma_{v,i}^2Ey_{\mathrm{lyap}}^{(i)}$ and $\hat{\mathcal{C}}_{\perp}^{(i)}=\hat{\mathcal{C}}_{\perp}^{(i-1)}+\gamma_{w,i}^2(z_{\mathrm{lyap}}^{(i)})^\top \hat{E}$.

8. Expand $V_{\mathrm{lyap}}^{(i)}=\begin{bmatrix}V_{\mathrm{lyap}}^{(i-1)}&v_{\mathrm{lyap}}^{(i)}\end{bmatrix}$, $C_{r,\mathrm{lyap}}^{(i)}=\begin{bmatrix}C_{r,\mathrm{lyap}}^{(i-1)}&Cv_{\mathrm{lyap}}^{(i)}\end{bmatrix}$, $S_{v,\mathrm{lyap}}^{(i)}=\begin{bmatrix}S_{v,\mathrm{lyap}}^{(i-1)}&(L_{v,\mathrm{lyap}}^{(i-1)})^\top l_{v,\mathrm{lyap}}^{(i)}\\0& s_{v,\mathrm{lyap}}^{(i)}\end{bmatrix}$, $L_{v,\mathrm{lyap}}^{(i)}=\begin{bmatrix}L_{v,\mathrm{lyap}}^{(i-1)}&l_{v,\mathrm{lyap}}^{(i)}\end{bmatrix}$, $\hat{W}_{\mathrm{lyap}}^{(i)}=\begin{bmatrix}\hat{W}_{\mathrm{lyap}}^{(i-1)}&w_{\mathrm{lyap}}^{(i)}\end{bmatrix}$, $\hat{B}_{r,\mathrm{lyap}}^{(i)}=\begin{bmatrix}\hat{B}_{r,\mathrm{lyap}}^{(i-1)}\\(w_{\mathrm{lyap}}^{(i)})^\top\hat{B}\end{bmatrix}$, \\ $S_{w,\mathrm{lyap}}^{(i)}=\begin{bmatrix}S_{w,\mathrm{lyap}}^{(i-1)}&(L_{w,\mathrm{lyap}}^{(i-1)})^\top l_{w,\mathrm{lyap}}^{(i)}\\0& s_{w,\mathrm{lyap}}^{(i)}\end{bmatrix}$, and $L_{w,\mathrm{lyap}}^{(i)}=\begin{bmatrix}L_{w,\mathrm{lyap}}^{(i-1)}&l_{w,\mathrm{lyap}}^{(i)}\end{bmatrix}$.

9. Compute $t_{v,\mathrm{lyap}}^{(i)}$ and $t_{w,\mathrm{lyap}}^{(i)}$ from \eqref{tv} and \eqref{tw}, and set $t_v^{(i)}=\begin{bmatrix}t_{v,1}^{(i)}\\t_{v,2}^{(i)}\end{bmatrix}=t_{v,\mathrm{lyap}}^{(i)}$ and $t_w^{(i)}=\begin{bmatrix}t_{w,1}^{(i)}\\t_{w,2}^{(i)}\end{bmatrix}=t_{w,\mathrm{lyap}}^{(i)}$.

10. Expand $T_{v}^{(i)}=\begin{bmatrix}T_{v}^{(i-1)}&t_{v,1}^{(i)}\\0&t_{v,2}^{(i)}\end{bmatrix}$ and $T_{w}^{(i)}=\begin{bmatrix}T_{w}^{(i-1)}&t_{w,1}^{(i)}\\0&t_{w,2}^{(i)}\end{bmatrix}$.

11. Set $v_i=V_{\mathrm{lyap}}^{(i)}t_v^{(i)}$ and $w_i=\hat{W}_{\mathrm{lyap}}^{(i)}t_w^{(i)}$, compute $x_i$ from \eqref{x_case1}.

12. Expand $\bar{X}^{(i)}=\mathrm{blkdiag}(\bar{X}^{(i-1)},x_i)$, $B_r^{(i)}=\begin{bmatrix}B_r^{(i-1)}\\-x_i\end{bmatrix}$, $\hat{B}_r^{(i)}=\begin{bmatrix}\hat{B}_r^{(i-1)}\\w_i^\top\hat{B}\end{bmatrix}$, $C_r^{(i)}=\begin{bmatrix}C_r^{(i-1)}&Cv_i\end{bmatrix}$, and $\hat{C}_r^{(i)}=\begin{bmatrix}\hat{C}_r^{(i)}&-x_i\end{bmatrix}$.

13. Update $B_{\perp}^{(i)}=B_{\perp}^{(i-1)}+Ev_ix_i$, $\hat{C}_{\perp}^{(i)}=\hat{C}_{\perp}^{(i-1)}+x_iw_i^\top\hat{E}$, $\tilde{K}^{(i)}=\tilde{K}^{(i-1)}+Ev_ix_iw_i^\top\hat{B}$, and

\hspace*{0.5cm}$\bar{K}^{(i)}=\bar{K}^{(i-1)}+Cv_ix_iw_i^\top\hat{E}$, $i=i+1$.\\

14. \textbf{else if} $\mathrm{Im}(\alpha_i)\neq0$ $\land$ $\mathrm{Im}(\beta_i)\neq0$ \textbf{then}\\

15. Set $\delta_{v,i}=\frac{\mathrm{Re}(\alpha_i)}{\mathrm{Im}(\alpha_i)}$, $v_{\mathrm{lyap}}^{(i)}=\begin{bmatrix}\sqrt{2}\gamma_{v,i}\Big(\mathrm{Re}\big(y_{\mathrm{lyap}}^{(i)}\big)+\delta_{v,i}\mathrm{Im}\big(y_{\mathrm{lyap}}^{(i)}\big)\Big)& \sqrt{2}\gamma_{v,i}\sqrt{\delta_{v,i}^2+1}\mathrm{Im}\big(y_{\mathrm{lyap}}^{(i)}\big)\end{bmatrix}$, $s_{v,\mathrm{lyap}}^{(i)}=\gamma_{v,i}^2\begin{bmatrix} 
I_{m} & \frac{\sqrt{1+\delta_{v,i}^2}}{2\delta_{v,i}}I_{m} \\
-\frac{\sqrt{1+\delta_{v,i}^2}}{2\delta_{v,i}}I_{m} & 0
\end{bmatrix}$, $l_{v,\mathrm{lyap}}^{(i)}=-\sqrt{2}\gamma_{v,i}\begin{bmatrix} 
I_{m} & 0
\end{bmatrix}$, $\delta_{w,i}=\frac{\mathrm{Re}(\beta_i)}{\mathrm{Im}(\beta_i)}$,

$w_{\mathrm{lyap}}^{(i)}=\begin{bmatrix}\sqrt{2}\gamma_{w,i}\Big(\mathrm{Re}\big(z_{\mathrm{lyap}}^{(i)}\big)+\delta_{w,i}\mathrm{Im}\big(z_{\mathrm{lyap}}^{(i)}\big)\Big)& \sqrt{2}\gamma_{w,i}\sqrt{\delta_{w,i}^2+1}\mathrm{Im}\big(z_{\mathrm{lyap}}^{(i)}\big)\end{bmatrix}$,\\
$s_{w,\mathrm{lyap}}^{(i)}=\gamma_{w,i}^2\begin{bmatrix} 
I_{m} & \frac{\sqrt{1+\delta_{w,i}^2}}{2\delta_{w,i}}I_{m} \\
-\frac{\sqrt{1+\delta_{w,i}^2}}{2\delta_{w,i}}I_{m} & 0
\end{bmatrix}$, $l_{w,\mathrm{lyap}}^{(i)}=-\sqrt{2}\gamma_{w,i}\begin{bmatrix} 
I_{m} & 0
\end{bmatrix}$.

16. Update $\mathcal{B}_{\perp}^{(i)}=\mathcal{B}_{\perp}^{(i-1)}+2\gamma_{v,i}^2E\Big(\mathrm{Re}\big(y_{\mathrm{lyap}}^{(i)}\big)+\delta_{v,i}\mathrm{Im}\big(y_{\mathrm{lyap}}^{(i)}\big)\Big)$ and $\hat{\mathcal{C}}_{\perp}^{(i)}=\hat{\mathcal{C}}_{\perp}^{(i-1)}+2\gamma_{w,i}^2\Big(\mathrm{Re}\big(z_{\mathrm{lyap}}^{(i)}\big)+\delta_{w,i}\mathrm{Im}\big(z_{\mathrm{lyap}}^{(i)}\big)\Big)^\top \hat{E}$.

17. Expand $V_{\mathrm{lyap}}^{(i)}=\begin{bmatrix}V_{\mathrm{lyap}}^{(i-1)}&v_{\mathrm{lyap}}^{(i)}\end{bmatrix}$, $C_{r,\mathrm{lyap}}^{(i)}=\begin{bmatrix}C_{r,\mathrm{lyap}}^{(i-1)}&Cv_{\mathrm{lyap}}^{(i)}\end{bmatrix}$, $S_{v,\mathrm{lyap}}^{(i)}=\begin{bmatrix}S_{v,\mathrm{lyap}}^{(i-1)}&(L_{v,\mathrm{lyap}}^{(i-1)})^\top l_{v,\mathrm{lyap}}^{(i)}\\0& s_{v,\mathrm{lyap}}^{(i)}\end{bmatrix}$, $L_{v,\mathrm{lyap}}^{(i)}=\begin{bmatrix}L_{v,\mathrm{lyap}}^{(i-1)}&l_{v,\mathrm{lyap}}^{(i)}\end{bmatrix}$, $\hat{W}_{\mathrm{lyap}}^{(i)}=\begin{bmatrix}\hat{W}_{\mathrm{lyap}}^{(i-1)}&w_{\mathrm{lyap}}^{(i)}\end{bmatrix}$, $\hat{B}_{r,\mathrm{lyap}}^{(i)}=\begin{bmatrix}\hat{B}_{r,\mathrm{lyap}}^{(i-1)}\\(w_{\mathrm{lyap}}^{(i)})^\top\hat{B}\end{bmatrix}$, \\ $S_{w,\mathrm{lyap}}^{(i)}=\begin{bmatrix}S_{w,\mathrm{lyap}}^{(i-1)}&(L_{w,\mathrm{lyap}}^{(i-1)})^\top l_{w,\mathrm{lyap}}^{(i)}\\0& s_{w,\mathrm{lyap}}^{(i)}\end{bmatrix}$, and $L_{w,\mathrm{lyap}}^{(i)}=\begin{bmatrix}L_{w,\mathrm{lyap}}^{(i-1)}&l_{w,\mathrm{lyap}}^{(i)}\end{bmatrix}$.

18. Compute $t_{v,\mathrm{lyap}}^{(i)}$ and $t_{w,\mathrm{lyap}}^{(i)}$ from \eqref{tv} and \eqref{tw}, and set $t_v^{(i)}=\begin{bmatrix}t_{v,1}^{(i)}\\t_{v,2}^{(i)}\end{bmatrix}=\begin{bmatrix}\mathrm{Re}\big(t_{v,\mathrm{lyap}}^{(i)}\big)&\mathrm{Im}\big(t_{v,\mathrm{lyap}}^{(i)}\big)\end{bmatrix}$ and $t_w^{(i)}=\begin{bmatrix}t_{w,1}^{(i)}\\t_{w,2}^{(i)}\end{bmatrix}=\begin{bmatrix}\mathrm{Re}\big(t_{w,\mathrm{lyap}}^{(i)}\big)&\mathrm{Im}\big(t_{w,\mathrm{lyap}}^{(i)}\big)\end{bmatrix}$.

19. Expand $T_{v}^{(i)}=\begin{bmatrix}T_{v}^{(i-1)}&t_{v,1}^{(i)}\\0&t_{v,2}^{(i)}\end{bmatrix}$ and $T_{w}^{(i)}=\begin{bmatrix}T_{w}^{(i-1)}&t_{w,1}^{(i)}\\0&t_{w,2}^{(i)}\end{bmatrix}$.

20. Set $v_i=\begin{bmatrix}\mathrm{Re}(y_i)&\mathrm{Im}(y_i)\end{bmatrix}=V_{\mathrm{lyap}}^{(i)}t_v^{(i)}$ and $w_i=\begin{bmatrix}\mathrm{Re}(z_i)&\mathrm{Im}(z_i)\end{bmatrix}=\hat{W}_{\mathrm{lyap}}^{(i)}t_w^{(i)}$, compute $x_i$ from \eqref{x_case2_f}-\eqref{x_case2_l}.

21. Expand $\bar{X}^{(i)}=\mathrm{blkdiag}(\bar{X}^{(i-1)},x_i)$, $B_r^{(i)}=\begin{bmatrix}B_r^{(i-1)}\\x_i\begin{bmatrix}-I_m\\0\end{bmatrix}\end{bmatrix}$, $\hat{B}_r^{(i)}=\begin{bmatrix}\hat{B}_r^{(i-1)}\\w_i^\top\hat{B}\end{bmatrix}$, $C_r^{(i)}=\begin{bmatrix}C_r^{(i-1)}&Cv_i\end{bmatrix}$, and $\hat{C}_r^{(i)}=\begin{bmatrix}\hat{C}_r^{(i)}&\begin{bmatrix}-I_m&0\end{bmatrix}x_i\end{bmatrix}$.

22. Update $B_{\perp}^{(i)}=B_{\perp}^{(i-1)}+Ev_ix_i$, $\hat{C}_{\perp}^{(i)}=\hat{C}_{\perp}^{(i-1)}+x_iw_i^\top\hat{E}$, $\tilde{K}^{(i)}=\tilde{K}^{(i-1)}+Ev_ix_iw_i^\top\hat{B}$, and

\hspace*{0.5cm}$\bar{K}^{(i)}=\bar{K}^{(i-1)}+Cv_ix_iw_i^\top\hat{E}$, $i=i+2$.\\

23. \textbf{else if} $\mathrm{Im}(\alpha_i)\neq0$ $\land$ $\mathrm{Im}(\beta_i)=0$ $\land$ $\mathrm{Im}(\beta_{i+1})=0$ \textbf{then}\\

24. Set $\delta_{v,i}=\frac{\mathrm{Re}(\alpha_i)}{\mathrm{Im}(\alpha_i)}$, $v_{\mathrm{lyap}}^{(i)}=\begin{bmatrix}\sqrt{2}\gamma_{v,i}\Big(\mathrm{Re}\big(y_{\mathrm{lyap}}^{(i)}\big)+\delta_{v,i}\mathrm{Im}\big(y_{\mathrm{lyap}}^{(i)}\big)\Big)& \sqrt{2}\gamma_{v,i}\sqrt{\delta_{v,i}^2+1}\mathrm{Im}\big(y_{\mathrm{lyap}}^{(i)}\big)\end{bmatrix}$, $s_{v,\mathrm{lyap}}^{(i)}=\gamma_{v,i}^2\begin{bmatrix} 
I_{m} & \frac{\sqrt{1+\delta_{v,i}^2}}{2\delta_{v,i}}I_{m} \\
-\frac{\sqrt{1+\delta_{v,i}^2}}{2\delta_{v,i}}I_{m} & 0
\end{bmatrix}$, and $l_{v,\mathrm{lyap}}^{(i)}=-\sqrt{2}\gamma_{v,i}\begin{bmatrix} 
I_{m} & 0
\end{bmatrix}$.

25. Update $\mathcal{B}_{\perp}^{(i)}=\mathcal{B}_{\perp}^{(i-1)}+2\gamma_{v,i}^2E\Big(\mathrm{Re}\big(y_{\mathrm{lyap}}^{(i)}\big)+\delta_{v,i}\mathrm{Im}\big(y_{\mathrm{lyap}}^{(i)}\big)\Big)$.

26. Expand $V_{\mathrm{lyap}}^{(i)}=\begin{bmatrix}V_{\mathrm{lyap}}^{(i-1)}&v_{\mathrm{lyap}}^{(i)}\end{bmatrix}$, $C_{r,\mathrm{lyap}}^{(i)}=\begin{bmatrix}C_{r,\mathrm{lyap}}^{(i-1)}&Cv_{\mathrm{lyap}}^{(i)}\end{bmatrix}$, $S_{v,\mathrm{lyap}}^{(i)}=\begin{bmatrix}S_{v,\mathrm{lyap}}^{(i-1)}&(L_{v,\mathrm{lyap}}^{(i-1)})^\top l_{v,\mathrm{lyap}}^{(i)}\\0& s_{v,\mathrm{lyap}}^{(i)}\end{bmatrix}$, and $L_{v,\mathrm{lyap}}^{(i)}=\begin{bmatrix}L_{v,\mathrm{lyap}}^{(i-1)}&l_{v,\mathrm{lyap}}^{(i)}\end{bmatrix}$.

27. Set $w_{\mathrm{lyap}}^{(i)}=\gamma_{w,i}z_{\mathrm{lyap}}^{(i)}$, $s_{w,\mathrm{lyap}}^{(i)}=-\beta_i I_m$, and $l_{w,\mathrm{lyap}}^{(i)}=-\gamma_{w,i} I_m$.

28. Update $\hat{\mathcal{C}}_{\perp}^{(i)}=\hat{\mathcal{C}}_{\perp}^{(i-1)}+\gamma_{w,i}^2(z_{\mathrm{lyap}}^{(i)})^\top \hat{E}$.

29. Expand $\hat{W}_{\mathrm{lyap}}^{(i)}=\begin{bmatrix}\hat{W}_{\mathrm{lyap}}^{(i-1)}&w_{\mathrm{lyap}}^{(i)}\end{bmatrix}$, $\hat{B}_{r,\mathrm{lyap}}^{(i)}=\begin{bmatrix}\hat{B}_{r,\mathrm{lyap}}^{(i-1)}\\(w_{\mathrm{lyap}}^{(i)})^\top\hat{B}\end{bmatrix}$, $S_{w,\mathrm{lyap}}^{(i)}=\begin{bmatrix}S_{w,\mathrm{lyap}}^{(i-1)}&(L_{w,\mathrm{lyap}}^{(i-1)})^\top l_{w,\mathrm{lyap}}^{(i)}\\0& s_{w,\mathrm{lyap}}^{(i)}\end{bmatrix}$, and $L_{w,\mathrm{lyap}}^{(i)}=\begin{bmatrix}L_{w,\mathrm{lyap}}^{(i-1)}&l_{w,\mathrm{lyap}}^{(i)}\end{bmatrix}$.

30. Solve for $z_{\mathrm{lyap}}^{(i)}$: $\big(\hat{A}^\top+\beta_{i+1} \hat{E}^\top\big)z_{\mathrm{lyap}}^{(i)}=\big(\hat{\mathcal{C}}_{\perp}^{(i)}\big)^\top$.

31. Set $\gamma_{w,i}=\sqrt{-2\mathrm{Re}(\beta_{i+1})}$ and $w_{\mathrm{lyap}}^{(i)}=\gamma_{w,i}z_{\mathrm{lyap}}^{(i)}$, $s_{w,\mathrm{lyap}}^{(i)}=-\beta_{i+1} I_m$, and $l_{w,\mathrm{lyap}}^{(i)}=-\gamma_{w,i} I_m$.

32. Update $\hat{\mathcal{C}}_{\perp}^{(i)}=\hat{\mathcal{C}}_{\perp}^{(i)}+\gamma_{w,i}^2(z_{\mathrm{lyap}}^{(i)})^\top \hat{E}$.

33. Expand $\hat{W}_{\mathrm{lyap}}^{(i)}=\begin{bmatrix}\hat{W}_{\mathrm{lyap}}^{(i-1)}&w_{\mathrm{lyap}}^{(i)}\end{bmatrix}$, $\hat{B}_{r,\mathrm{lyap}}^{(i)}=\begin{bmatrix}\hat{B}_{r,\mathrm{lyap}}^{(i-1)}\\(w_{\mathrm{lyap}}^{(i)})^\top\hat{B}\end{bmatrix}$, $S_{w,\mathrm{lyap}}^{(i)}=\begin{bmatrix}S_{w,\mathrm{lyap}}^{(i-1)}&(L_{w,\mathrm{lyap}}^{(i-1)})^\top l_{w,\mathrm{lyap}}^{(i)}\\0& s_{w,\mathrm{lyap}}^{(i)}\end{bmatrix}$, and $L_{w,\mathrm{lyap}}^{(i)}=\begin{bmatrix}L_{w,\mathrm{lyap}}^{(i-1)}&l_{w,\mathrm{lyap}}^{(i)}\end{bmatrix}$.

34. Compute $t_{v,\mathrm{lyap}}^{(i)}$, $t_{w,\mathrm{lyap}}^{(i)}$, and $t_{w,\mathrm{lyap}}^{\prime(i)}$ from \eqref{tv}, \eqref{tw}, and \eqref{twp}, and set\\ $t_v^{(i)}=\begin{bmatrix}t_{v,1}^{(i)}\\t_{v,2}^{(i)}\end{bmatrix}=\begin{bmatrix}\mathrm{Re}\big(t_{v,\mathrm{lyap}}^{(i)}\big)&\mathrm{Im}\big(t_{v,\mathrm{lyap}}^{(i)}\big)\end{bmatrix}$ and $t_w^{(i)}=\begin{bmatrix}t_{w,1}^{(i)}\\t_{w,2}^{(i)}\end{bmatrix}=\begin{bmatrix}t_{w,\mathrm{lyap}}^{(i)}&t_{w,\mathrm{lyap}}^{\prime(i)}\end{bmatrix}$.

35. Set $v_i=\begin{bmatrix}\mathrm{Re}(y_i)&\mathrm{Im}(y_i)\end{bmatrix}=V_{\mathrm{lyap}}^{(i)}t_v^{(i)}$ and $w_i=\begin{bmatrix}z_i&z_i^{\prime}\end{bmatrix}=\hat{W}_{\mathrm{lyap}}^{(i)}t_w^{(i)}$, and compute $x_i$ from \eqref{x_case3_f}-\eqref{x_case3_l}.

36. Expand $\bar{X}^{(i)}=\mathrm{blkdiag}(\bar{X}^{(i-1)},x_i)$, $B_r^{(i)}=\begin{bmatrix}B_r^{(i-1)}\\x_i\begin{bmatrix}-I_m\\0\end{bmatrix}\end{bmatrix}$, $\hat{B}_r^{(i)}=\begin{bmatrix}\hat{B}_r^{(i-1)}\\w_i^\top\hat{B}\end{bmatrix}$, $C_r^{(i)}=\begin{bmatrix}C_r^{(i-1)}&Cv_i\end{bmatrix}$, and $\hat{C}_r^{(i)}=\begin{bmatrix}\hat{C}_r^{(i)}&\begin{bmatrix}-I_m&0\end{bmatrix}x_i\end{bmatrix}$.

37. Update $B_{\perp}^{(i)}=B_{\perp}^{(i-1)}+Ev_ix_i\begin{bmatrix}I_m\\0\end{bmatrix}$, $\hat{C}_{\perp}^{(i)}=\hat{C}_{\perp}^{(i-1)}+\begin{bmatrix}I_m&0\end{bmatrix}x_iw_i^\top\hat{E}$, $\tilde{K}^{(i)}=\tilde{K}^{(i-1)}+Ev_ix_iw_i^\top\hat{B}$,

\hspace*{0.5cm}and $\bar{K}^{(i)}=\bar{K}^{(i-1)}+Cv_ix_iw_i^\top\hat{E}$, $i=i+2$.\\

38. \textbf{else if} $\mathrm{Im}(\alpha_i)=0$ $\land$ $\mathrm{Im}(\alpha_{i+1})=0$ $\land$ $\mathrm{Im}(\beta_{i})\neq0$ \textbf{then}\\

39. Set $v_{\mathrm{lyap}}^{(i)}=\gamma_{v,i}y_{\mathrm{lyap}}^{(i)}$, $s_{v,\mathrm{lyap}}^{(i)}=-\alpha_i I_m$, and $l_{v,\mathrm{lyap}}^{(i)}=-\gamma_{v,i} I_m$.

40. Update $\mathcal{B}_{\perp}^{(i)}=\mathcal{B}_{\perp}^{(i-1)}+\gamma_{v,i}^2Ey_{\mathrm{lyap}}^{(i)}$.

41. Expand $V_{\mathrm{lyap}}^{(i)}=\begin{bmatrix}V_{\mathrm{lyap}}^{(i-1)}&v_{\mathrm{lyap}}^{(i)}\end{bmatrix}$, $C_{r,\mathrm{lyap}}^{(i)}=\begin{bmatrix}C_{r,\mathrm{lyap}}^{(i-1)}&Cv_{\mathrm{lyap}}^{(i)}\end{bmatrix}$, $S_{v,\mathrm{lyap}}^{(i)}=\begin{bmatrix}S_{v,\mathrm{lyap}}^{(i-1)}&(L_{v,\mathrm{lyap}}^{(i-1)})^\top l_{v,\mathrm{lyap}}^{(i)}\\0& s_{v,\mathrm{lyap}}^{(i)}\end{bmatrix}$, and $L_{v,\mathrm{lyap}}^{(i)}=\begin{bmatrix}L_{v,\mathrm{lyap}}^{(i-1)}&l_{v,\mathrm{lyap}}^{(i)}\end{bmatrix}$.

42. Solve for $y_{\mathrm{lyap}}^{(i)}$: $\big(A+\alpha_{i+1} E\big)y_{\mathrm{lyap}}^{(i)}=\mathcal{B}_{\perp}^{(i)}$.

43. Set $\gamma_{v,i}=\sqrt{-2\mathrm{Re}(\alpha_{i+1})}$, $v_{\mathrm{lyap}}^{(i)}=\gamma_{v,i}y_{\mathrm{lyap}}^{(i)}$, $s_{v,\mathrm{lyap}}^{(i)}=-\alpha_{i+1} I_m$, and $l_{v,\mathrm{lyap}}^{(i)}=-\gamma_{v,i} I_m$.

44. Update $\mathcal{B}_{\perp}^{(i)}=\mathcal{B}_{\perp}^{(i)}+\gamma_{v,i}^2Ey_{\mathrm{lyap}}^{(i)}$.

45. Expand $V_{\mathrm{lyap}}^{(i)}=\begin{bmatrix}V_{\mathrm{lyap}}^{(i-1)}&v_{\mathrm{lyap}}^{(i)}\end{bmatrix}$, $C_{r,\mathrm{lyap}}^{(i)}=\begin{bmatrix}C_{r,\mathrm{lyap}}^{(i-1)}&Cv_{\mathrm{lyap}}^{(i)}\end{bmatrix}$, $S_{v,\mathrm{lyap}}^{(i)}=\begin{bmatrix}S_{v,\mathrm{lyap}}^{(i-1)}&(L_{v,\mathrm{lyap}}^{(i-1)})^\top l_{v,\mathrm{lyap}}^{(i)}\\0& s_{v,\mathrm{lyap}}^{(i)}\end{bmatrix}$, and $L_{v,\mathrm{lyap}}^{(i)}=\begin{bmatrix}L_{v,\mathrm{lyap}}^{(i-1)}&l_{v,\mathrm{lyap}}^{(i)}\end{bmatrix}$.

46. Set $\delta_{w,i}=\frac{\mathrm{Re}(\beta_i)}{\mathrm{Im}(\beta_i)}$, $w_{\mathrm{lyap}}^{(i)}=\begin{bmatrix}\sqrt{2}\gamma_{w,i}\Big(\mathrm{Re}\big(z_{\mathrm{lyap}}^{(i)}\big)+\delta_{w,i}\mathrm{Im}\big(z_{\mathrm{lyap}}^{(i)}\big)\Big)& \sqrt{2}\gamma_{w,i}\sqrt{\delta_{w,i}^2+1}\mathrm{Im}\big(z_{\mathrm{lyap}}^{(i)}\big)\end{bmatrix}$, $s_{w,\mathrm{lyap}}^{(i)}=\gamma_{w,i}^2\begin{bmatrix} 
I_{m} & \frac{\sqrt{1+\delta_{w,i}^2}}{2\delta_{w,i}}I_{m} \\
-\frac{\sqrt{1+\delta_{w,i}^2}}{2\delta_{w,i}}I_{m} & 0
\end{bmatrix}$, $l_{w,\mathrm{lyap}}^{(i)}=-\sqrt{2}\gamma_{w,i}\begin{bmatrix} 
I_{m} & 0
\end{bmatrix}$.

47. Update $\hat{\mathcal{C}}_{\perp}^{(i)}=\hat{\mathcal{C}}_{\perp}^{(i-1)}+2\gamma_{w,i}^2\Big(\mathrm{Re}\big(z_{\mathrm{lyap}}^{(i)}\big)+\delta_{w,i}\mathrm{Im}\big(z_{\mathrm{lyap}}^{(i)}\big)\Big)^\top \hat{E}$.

48. $\hat{W}_{\mathrm{lyap}}^{(i)}=\begin{bmatrix}\hat{W}_{\mathrm{lyap}}^{(i-1)}&w_{\mathrm{lyap}}^{(i)}\end{bmatrix}$, $\hat{B}_{r,\mathrm{lyap}}^{(i)}=\begin{bmatrix}\hat{B}_{r,\mathrm{lyap}}^{(i-1)}\\(w_{\mathrm{lyap}}^{(i)})^\top\hat{B}\end{bmatrix}$, $S_{w,\mathrm{lyap}}^{(i)}=\begin{bmatrix}S_{w,\mathrm{lyap}}^{(i-1)}&(L_{w,\mathrm{lyap}}^{(i-1)})^\top l_{w,\mathrm{lyap}}^{(i)}\\0& s_{w,\mathrm{lyap}}^{(i)}\end{bmatrix}$, and $L_{w,\mathrm{lyap}}^{(i)}=\begin{bmatrix}L_{w,\mathrm{lyap}}^{(i-1)}&l_{w,\mathrm{lyap}}^{(i)}\end{bmatrix}$.

49. Compute $t_{v,\mathrm{lyap}}^{(i)}$, $t_{v,\mathrm{lyap}}^{\prime(i)}$, and $t_{w,\mathrm{lyap}}^{(i)}$ from \eqref{tv}, \eqref{tvp}, and \eqref{tw}, and set 
$t_v^{(i)}=\begin{bmatrix}t_{v,1}^{(i)}\\t_{v,2}^{(i)}\end{bmatrix}=\begin{bmatrix}t_{v,\mathrm{lyap}}^{(i)}&t_{v,\mathrm{lyap}}^{\prime(i)}\end{bmatrix}$ and
$t_w^{(i)}=\begin{bmatrix}t_{w,1}^{(i)}\\t_{w,2}^{(i)}\end{bmatrix}=\begin{bmatrix}\mathrm{Re}\big(t_{w,\mathrm{lyap}}^{(i)}\big)&\mathrm{Im}\big(t_{w,\mathrm{lyap}}^{(i)}\big)\end{bmatrix}$.

50. Expand $T_{v}^{(i)}=\begin{bmatrix}T_{v}^{(i-1)}&t_{v,1}^{(i)}\\0&t_{v,2}^{(i)}\end{bmatrix}$ and $T_{w}^{(i)}=\begin{bmatrix}T_{w}^{(i-1)}&t_{w,1}^{(i)}\\0&t_{w,2}^{(i)}\end{bmatrix}$.

51. Set $v_i=\begin{bmatrix}y_i&y_i^{\prime}\end{bmatrix}$ and $w_i=\begin{bmatrix}\mathrm{Re}(z_i)&\mathrm{Im}(z_i)\end{bmatrix}$, and compute $x_i$ from \eqref{x_case4_f}-\eqref{x_case4_l}.

52. Expand $\bar{X}^{(i)}=\mathrm{blkdiag}(\bar{X}^{(i-1)},x_i)$, $B_r^{(i)}=\begin{bmatrix}B_r^{(i-1)}\\x_i\begin{bmatrix}-I_m\\0\end{bmatrix}\end{bmatrix}$, $\hat{B}_r^{(i)}=\begin{bmatrix}\hat{B}_r^{(i-1)}\\w_i^\top\hat{B}\end{bmatrix}$, $C_r^{(i)}=\begin{bmatrix}C_r^{(i-1)}&Cv_i\end{bmatrix}$, and $\hat{C}_r^{(i)}=\begin{bmatrix}\hat{C}_r^{(i)}&\begin{bmatrix}-I_m&0\end{bmatrix}x_i\end{bmatrix}$.

53. Update $B_{\perp}^{(i)}=B_{\perp}^{(i-1)}+Ev_ix_i\begin{bmatrix}I_m\\0\end{bmatrix}$, $\hat{C}_{\perp}^{(i)}=\hat{C}_{\perp}^{(i-1)}+\begin{bmatrix}I_m&0\end{bmatrix}x_iw_i^\top\hat{E}$, $\tilde{K}^{(i)}=\tilde{K}^{(i-1)}+Ev_ix_iw_i^\top\hat{B}$,

\hspace*{0.5cm}and $\bar{K}^{(i)}=\bar{K}^{(i-1)}+Cv_ix_iw_i^\top\hat{E}$, $i=i+2$.

54. \textbf{end if}

55. \textbf{end while}

\noindent\rule{\textwidth}{0.8pt}
\subsection{Automatic Shift Generation}
The Lyapunov equation, Sylvester equation, and symmetric Riccati equation can all be seen as special cases of the NARE. Consequently, N-RADI is arguably the most general low-rank ADI algorithm, as the low-rank ADI methods in \cite{benner2014computing,benner2013reformulated,zulfiqar2026ldl,benner2013efficient,benner2009adi,benner2018radi} are special cases of N-RADI. We now discuss automatic shift generation in N-RADI, which will essentially cover shift generation in all other low-rank ADI algorithms. The discussion in the sequel closely follows that in \cite{zulfiqar2025unified}. For simplicity, in the subsequent discussion we assume that $E^{-1}A$ and $\hat{A}\hat{E}^{-1}$ are Hurwitz.

Let us define $G_{\perp}^{(i)}(s)$ and $\hat{G}_{\perp}^{(i)}(s)$ as follows:
\begin{align}
G_{\perp}^{(i)}(s)=C(sE-A)^{-1}B_{\perp}^{(i)}\quad \text{and}\quad \hat{G}_{\perp}^{(i)}(s)=\hat{C}_{\perp}^{(i)}(s\hat{E}-\hat{A})^{-1}\hat{B}.
\end{align}
It can readily be noted that
\begin{align}
G_{\perp}^{(i)}(s)\hat{G}_{\perp}^{(i)}(s)=C(sE-A)^{-1}R_x^{(i)}(s\hat{E}-\hat{A})^{-1}\hat{B}.\label{obsv1}
\end{align}
Then the errors $G(s)-G_r^{(i)}(s)$ and $\hat{G}(s)-\hat{G}_r^{(i)}(s)$ can be factorized as follows:
\begin{align}
G(s)-G_r^{(i)}(s)&=G_{\perp}^{(i)}(s)\Big(-L_v^{(i)}\big(sI-A_r^{(i)}\big)^{-1}B_r^{(i)}+I\Big),\\
\hat{G}(s)-\hat{G}_r^{(i)}(s)&=\Big(-\hat{C}_r^{(i)}\big(sI-\hat{A}_r^{(i)}\big)^{-1}(L_w^{(i)})^\top+I\Big)\hat{G}_{\perp}^{(i)}(s).
\end{align}
It is shown in Lemma 3.3 of \cite{wolfthesis} that these errors are zero when all the poles of the realization $(E,A,B_{\perp}^{(i)},C)$ are uncontrollable and all the poles of the realization $(\hat{E},\hat{A},\hat{B},\hat{C}_{\perp}^{(i)})$ are unobservable. N-RADI is essentially a recursive interpolation algorithm that simultaneously solves two different model order reduction problems. The peaks in the frequency domain plot of $G(s)$ associated with the most controllable modes are flattened out in $G_{\perp}^{(i)}(s)$. Note that $G(s)$ and $G_{\perp}^{(i)}(s)$ have the same poles. When the most controllable poles of $G(s)$ are flattened out in $G_{\perp}^{(i)}(s)$, the error $G(s)-G_r^{(i)}(s)$ drops significantly and all the poles of $G_{\perp}^{(i)}(s)$ become poorly controllable. In the eigenvalue problem literature \cite{rommes2007methods,saad2011numerical}, replacing $B$ with $B_{\perp}^{(i)}$ to obtain $G_{\perp}^{(i)}(s)$ is called deflation, which is used in dominant poles estimation algorithms \cite{rommes2006efficient,rommes2008convergence,martins1996computing,mengi2022large} to stop targeting the dominant poles that are already captured, since such poles become poorly controllable in $G_{\perp}^{(i)}(s)$.

It can be noted from \eqref{obsv1} that a small residual $R_x^{(i)}$ and a small norm of $G_{\perp}^{(i)}(s)\hat{G}_{\perp}^{(i)}(s)$ are directly related. From the analysis of the weighted $\mathcal{H}_2$ norm in \cite{anic2013interpolatory,breiten2015near}, note that the $\mathcal{H}_2$ norm of $G_{\perp}^{(i)}(s)\hat{G}_{\perp}^{(i)}(s)$ is dominated by the most controllable poles of $G_{\perp}^{(i)}(s)$ and the most observable poles of $\hat{G}_{\perp}^{(i)}(s)$. Also, note that a peak in the frequency domain of $G(s)$ which is flattened out in $G_{\perp}^{(i)}(s)$ can be emphasized by a peak in $\hat{G}_{\perp}^{(i)}(s)$ if this peak is not captured in $\hat{G}_r(s)$. Similarly, a peak in the frequency domain of $\hat{G}(s)$ which is flattened out in $\hat{G}_{\perp}^{(i)}(s)$ can be emphasized by a peak in $G_{\perp}^{(i)}(s)$ if this peak is not captured in $G_r^{(i)}(s)$. Thus it is important to interpolate $G(s)$ at the mirror images of its most controllable poles as well as the most observable poles of $\hat{G}(s)$; see \cite{zulfiqar2025unified} for details. To avoid this situation, one can restrict the shifts to $\alpha_i=\beta_i$ as done in \cite{zulfiqar2025unified} for Sylvester equations.

The interpolatory behavior of subspace accelerated dominant pole estimation (SADPA) \cite{rommes2006efficient} and low-rank ADI methods is very similar \cite{zulfiqar2025unified,mengi2022large}. The main difference is that SADPA applies deflation only after capturing the peak associated with dominant poles, whereas low-rank ADI methods deflate in every iteration. This premature deflation results in inferior approximation of the dominant poles as reported in \cite{zulfiqar2025unified}. However, this is not an issue as such since the main goal of low-rank ADI methods is to reduce the residual of the matrix equation, not to capture dominant poles. Nevertheless, we can automatically generate the ADI shifts similarly to how shifts are generated automatically in SADPA \cite{rommes2006efficient}, as done in \cite{zulfiqar2025unified}.

We propose setting  
\[
V_{\mathrm{proj}} = \mathrm{orth}\big(\begin{bmatrix} v_1 & \cdots & v_i \end{bmatrix}\big)
\]  
with an implicit restart mechanism used in SADPA \cite{rommes2006efficient}, i.e., if the number of columns of $V_{\mathrm{proj}}$ exceeds a set limit, the previous history of $v_i$ is discarded and new history is collected. Implicit restart is used in several eigenvalue solvers \cite{saad2011numerical} to keep the number of columns of $V_{\mathrm{proj}}$ within an acceptable limit.

Next, project as follows  
\[
E_{\mathrm{proj}} = V_{\mathrm{proj}}^\top E V_{\mathrm{proj}}, \quad
A_{\mathrm{proj}} = V_{\mathrm{proj}}^\top A V_{\mathrm{proj}}, \quad
B_{\mathrm{proj}} = V_{\mathrm{proj}}^\top B_{\perp}^{(i)}.
\]  
Compute the eigenvalue decomposition of $E_{\mathrm{proj}}^{-1}A_{\mathrm{proj}}$ as  
$E_{\mathrm{proj}}^{-1}A_{\mathrm{proj}} = T\,\mathrm{diag}\big(\lambda_1,\dots,\lambda_k\big)T^{-1}$. Define $r_{b,l} = T^{-1}(l,:)\,B_{\mathrm{proj}}$.  
The most controllable pole of $E_{\mathrm{proj}}^{-1}A_{\mathrm{proj}}$ is the pole $\lambda_l$ associated with the largest value of  
\[
\phi_l = \frac{\|r_{b,l}\|_2^2}{|\mathrm{Re}(\lambda_l)|},
\]  
as in \cite{rommes2007methods,mengi2022large}. Then use the most controllable pole of these projected matrices as the ADI shifts $\alpha_i$ and $\beta_i$, ensuring $\alpha_i = \beta_i$.

In the subsequent iteration, select  
\[
\hat{W}_{\mathrm{proj}} = \mathrm{orth}\big(\begin{bmatrix} w_1 & \cdots & w_i \end{bmatrix}\big)
\]  
with implicit restart, and project as:  
\[
\hat{E}_{\mathrm{proj}} = \hat{W}_{\mathrm{proj}}^\top \hat{E} \hat{W}_{\mathrm{proj}}, \quad
\hat{A}_{\mathrm{proj}} = \hat{W}_{\mathrm{proj}}^\top \hat{A} \hat{W}_{\mathrm{proj}}, \quad
\hat{C}_{\mathrm{proj}} = \hat{C}_{\perp}^{(i)} \hat{W}_{\mathrm{proj}}.
\]  
Compute the eigenvalue decomposition of $\hat{A}_{\mathrm{proj}}\hat{E}_{\mathrm{proj}}^{-1}$:  
$\hat{A}_{\mathrm{proj}}\hat{E}_{\mathrm{proj}}^{-1} = \hat{T}\,\mathrm{diag}\big(\hat{\lambda}_1,\dots,\hat{\lambda}_r\big)\hat{T}^{-1}$.  
Define $\hat{r}_{c,l} = \hat{C}_{\mathrm{proj}}\,\hat{T}(:,l)$.  
The most observable pole of $\hat{A}_{\mathrm{proj}}\hat{E}_{\mathrm{proj}}^{-1}$ is the pole $\hat{\lambda}_l$ corresponding to the largest  
\[
\hat{\phi}_l = \frac{\|\hat{r}_{c,l}\|_2^2}{|\mathrm{Re}(\hat{\lambda}_l)|},
\]  
as in \cite{rommes2007methods,mengi2022large}. Use the most observable pole of these projected matrices as the common ADI shift $\alpha_i = \beta_i$.
\subsection{Sample MATLAB-based Implementations}
The sample MATLAB-based implementations of N-RADI and UN-RADI are given in the appendices. Both implementations are equipped with the subspace acceleration-based shift generation strategy described in the previous subsection. If the user does not provide any precomputed shifts, the MATLAB functions generate the ADI shifts automatically once an initial first shift is provided by the user. If the appropriate flag is set to $1$, the implementation of N-RADI also produces $S_v^{(i)}$, $L_v^{(i)}$, $S_w^{(i)}$, and $L_w^{(i)}$ for verifying various mathematical properties mentioned in this section.
\section{Numerical Results}
The numerical performance of the proposed algorithms is evaluated using semi-discretized heat transfer models for optimal cooling of steel profiles \cite{benner2005semi}, also referred to as the rail models in the literature. The matrices involved in the NARE have the following dimensions: $E\in\mathbb{R}^{12,65,537\times 12,65,537}$, $A\in\mathbb{R}^{12,65,537\times 12,65,537}$, $B\in\mathbb{R}^{12,65,537\times7}$, $C\in\mathbb{R}^{6\times 12,65,537}$, $\hat{E}\in\mathbb{R}^{3,17,377\times 3,17,377}$, $\hat{A}\in\mathbb{R}^{3,17,377\times 3,17,377}$, $\hat{B}\in\mathbb{R}^{3,17,377\times 6}$, and $\hat{C}\in\mathbb{R}^{7\times 3,17,377}$. MATLAB codes to reproduce the results are publicly available at \cite{mycode}. The initial shifts in N-RADI and UN-RADI are set to $\alpha_1=-0.001$ and $\beta_1=-0.001$. The ADI shifts are generated automatically using the subspace accelerated strategy described in the previous subsection. The maximum number of columns in $V_{\mathrm{proj}}$ and $W_{\mathrm{proj}}$ is set to $14$, after which implicit restart is applied to keep the eigenvalue decomposition for the shifts at order $14$. The maximum allowed number of iterations is $100$, and the convergence tolerance $\tau$ is set to $10^{-10}$. All tests are performed using MATLAB R2025b on a Windows 11 laptop with 32 GB of RAM and an Intel(R) Core(TM) Ultra 9 285H 2.9 GHz processor.

N-RADI converged in 57 iterations taking 274.7042 seconds, while UN-RADI required 376.7226 seconds. In this test, avoiding the SMW formula did not reduce the overall computational cost. However, as $p$ increases, UN-RADI is expected to outperform N-RADI in terms of computational efficiency. The decay of the residual is shown in Figure \ref{fig1}.
\begin{figure}[!h]
  \centering
  \includegraphics[width=16cm]{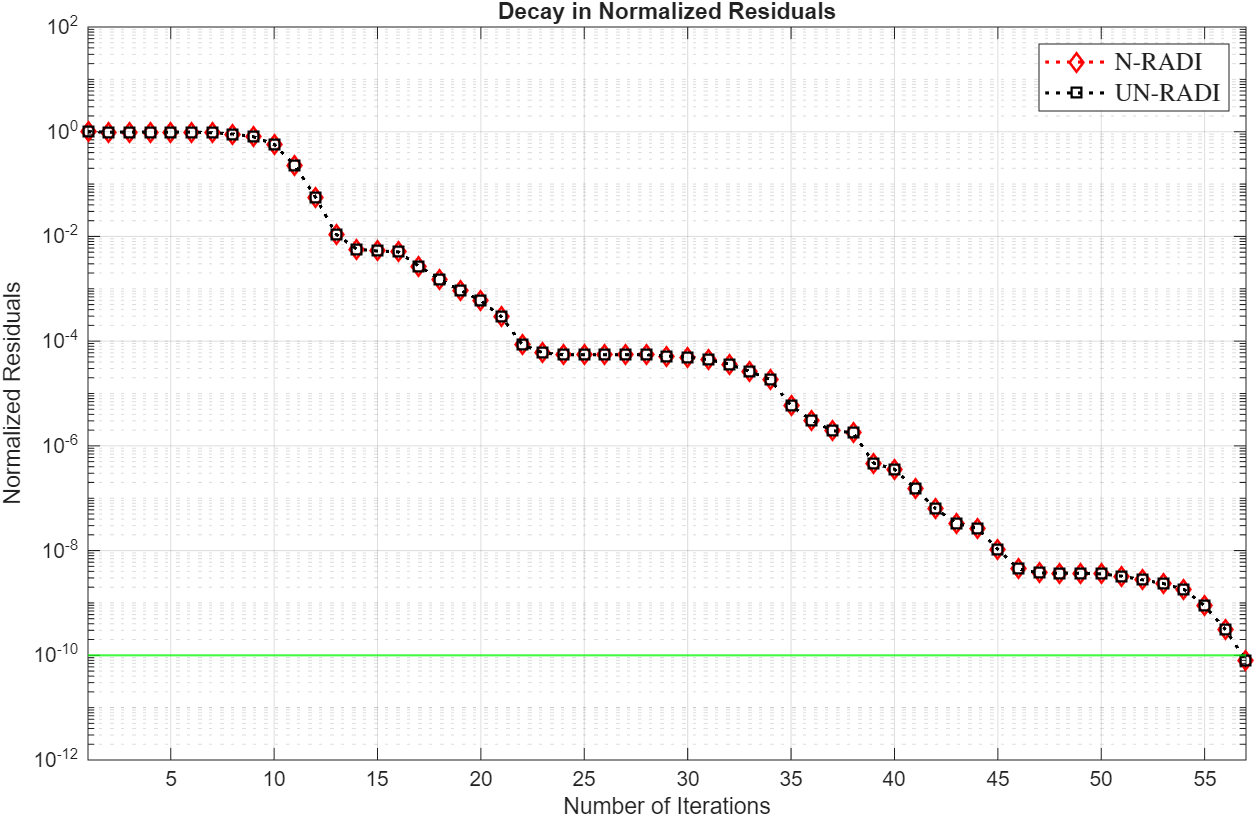}
  \caption{Normalized residual for \(X\)}\label{fig1}
\end{figure}
\section{Conclusion}
A low-rank ADI algorithm for solving large-scale NAREs is presented. The mathematical properties and residual expression for tracking the quality of approximation in large-scale setting are also discussed. Two implementations of the algorithm are presented. One uses the SMW formula while the other uses the low-rank ADI algorithm for Lyapunov equations as a base algorithm to avoid the SMW formula and extract the low-rank factors of the solution to the NAREs from the low-rank factors of the solution to the Lyapunov equations. A self-generating shift strategy is also proposed that automatically generates shifts without any user interference. A numerical example of order \(10^6\) comprising benchmark data is presented which shows that the proposed algorithm is computationally efficient and accurate when the proposed automatic shift generation is used.
\section*{Appendix A: MATLAB Implementation of N-RADI}
\begin{verbatim}
function [V,W,X,K,Kh,Res,a_used,b_used,Sv,Lv,Sw,Lw] = ...
         N_RADI(E,A,B,C,Eh,Ah,Bh,Ch,a,b,a_in,b_in,tol,kmax,rank_max,flag_s)

% Low-rank ADI Algorithm for Large-scale Non-symmetric Algebraic Riccati
% Equations (NARE)
% It solves the NARE 
% A Q Eh + E Q Ah - E Q Bh C Q Eh + B Ch =0
% Q \approx V X W'
% K \approx K_gain = E Q Bh
% Kh \approx Kh_gain = C Q Eh
% Author: Umair Zulfiqar
% Date: 20th April 2026


if any(real(a) >= -1e-8) || any(real(a_in) >= -1e-8)
    disp('Error: All the ADI shifts must have negative real parts')
    disp('Fix it and try again!')
    return
end


%% Initialization
[n,m]=size(B); p=size(C,1); nh=size(Bh,1); Im=eye(m);

if ~isempty(a) || ~isempty(b)
    kmax = length(a);
else
    a=a_in; b=b_in;
end
max_cols = (kmax+2) * m;
V = zeros(n, max_cols);
W = zeros(nh, max_cols);
V_proj=[]; W_proj=[];
X = zeros(max_cols, max_cols);

B_=B; C_ = Ch;
col_idx = 1;
k = 1; r_idx_v = 1; r_idx_w = 1;
Sv=[]; Lv=[]; Sw=[]; Lw=[];
K=zeros(n,p); Kh=zeros(p,nh);
flag=1; res=1; Res=[]; itr=1; c_alt=1;
%% Iterations

while flag

    if itr>k
        disp('ADI Shifts are not ordered properly')
        flag=0;
        break
    end
    
    Iteration_No=itr

    if k>kmax
        disp('Maximum iteration Reached')
        flag=0;
        break
    end

    if res<tol
        disp('N-RADI Converged')
        flag=0;
        break
    end
   
    ak = a(k); bk = b(k);

    if k<length(a)
        ak1=a(k+1); bk1=b(k+1);
    end
   
    if k==1
        vk=(A+ak*E)\B_;
        wk=(Ah'+bk*Eh')\(C_');
    else
        vk = solve_shifted_smw(A, E, K, C, B_, p, ak);
        wk = solve_shifted_smw(Ah', Eh', Kh', Bh', C_', p, bk);
    end
    
    if flag_s
        curr_cols = col_idx - 1;
        V_curr = V(:, 1:curr_cols);
        W_curr = W(:, 1:curr_cols);
        X_curr = X(1:curr_cols, 1:curr_cols);
    end

    if isreal(ak) && isreal(bk)

        disp('Case 1')
        
        if flag_s
            sv=-ak*Im; lv=-Im; sw=-bk*Im; lw=-Im;
            if k==1
                Sv=blkdiag(Sv,sv); Lv=[Lv lv];
                Sw=blkdiag(Sw,sw); Lw=[Lw lw];
            else
                Sv=[Sv X_curr*(Lw'*lv+(W_curr'*Bh)*(C*vk));
                    zeros(m,size(Sv,2)) sv];
                Sw=[Sw X_curr'*(Lv'*lw+(V_curr'*C')*(Bh'*wk));
                    zeros(m,size(Sw,2)) sw];
                Lv=[Lv lv]; Lw=[Lw lw];
            end
        end

        bhr=wk'*Bh; cr=C*vk;
        
        x=-(ak+bk)*inv(Im+bhr*cr);
        K=K+(E*vk)*x*bhr; Kh=Kh+cr*x*(wk'*Eh);
        
        block_size = m;
        idx_range = col_idx : col_idx + block_size - 1;
        V(:, idx_range) = vk;
        W(:, idx_range) = wk;
        X(idx_range, idx_range) = x;
        B_=B_+(E*vk)*x; C_=C_+x*(wk'*Eh);
        res=max(sqrt(eig(full(B_'*B_*C_*C_'))))/...
                                        max(sqrt(eig(full(B'*B*Ch*Ch'))));
        Residual=res
        Res=[Res res];
        col_idx = col_idx + block_size;
        k = k + 1;
        
    elseif ~isreal(ak) && ~isreal(bk)

        disp('Case 2')

        ar = real(ak); ai = imag(ak);
        br = real(bk); bi = imag(bk);
        vr = real(vk); vi = imag(vk);
        wr = real(wk); wi = imag(wk);
        
        if flag_s
            sv=kron(-[ar ai; -ai ar],Im); lv=kron([-1 0],Im);
            sw=kron(-[br bi; -bi br],Im); lw=kron([-1 0],Im);
            if k==1
                Sv=blkdiag(Sv,sv); Lv=[Lv lv];
                Sw=blkdiag(Sw,sw); Lw=[Lw lw];
            else
                Sv=[Sv X_curr*(Lw'*lv+(W_curr'*Bh)*(C*[vr,vi]));
                    zeros(2*m,size(Sv,2)) sv];
                Sw=[Sw X_curr'*(Lv'*lw+(V_curr'*C')*(Bh'*[wr,wi]));
                    zeros(2*m,size(Sw,2)) sw];
                Lv=[Lv lv]; Lw=[Lw lw];
            end
        end
        

        bhr=[wr,wi]'*Bh; cr=C*[vr,vi];

        x=compute_x_1(ar,ai,br,bi,bhr(1:m,:),...
                                bhr(m+1:2*m,:),cr(:,1:m),cr(:,m+1:2*m),Im);
        K=K+(E*[vr,vi])*x*bhr; Kh=Kh+cr*x*([wr,wi]'*Eh);

        block_size = 2 * m;
        idx_range = col_idx : col_idx + block_size - 1;
        V(:, idx_range) = [vr, vi];
        W(:, idx_range) = [wr, wi];
        X(idx_range, idx_range) = x;
        B_=B_+(E*[vr,vi])*x*[Im; 0*Im]; C_=C_+[Im,0*Im]*x*([wr,wi]'*Eh);
        res=max(sqrt(eig(full(B_'*B_*C_*C_'))))/...
                                         max(sqrt(eig(full(B'*B*Ch*Ch'))));
        Residual=res
        Res=[Res res];
        col_idx = col_idx + block_size;
        k = k + 2;

    elseif k<length(a) && ~isreal(ak) && isreal(bk) && isreal(bk1)

        disp('Case 3')

        wk1 = solve_shifted_smw(Ah', Eh', Kh', Bh', Eh'*wk, p, bk1);

        ar = real(ak); ai = imag(ak);
        vr = real(vk); vi = imag(vk);
        wr = wk; wi = wk1;
        
        if flag_s
            sv=kron(-[ar ai; -ai ar],Im); lv=kron([-1 0],Im);
            sw=kron([-bk 1; 0 -bk1],Im); lw=kron([-1 0],Im);
            if k==1
                Sv=blkdiag(Sv,sv); Lv=[Lv lv];
                Sw=blkdiag(Sw,sw); Lw=[Lw lw];
            else
                Sv=[Sv X_curr*(Lw'*lv+(W_curr'*Bh)*(C*[vr,vi]));
                    zeros(2*m,size(Sv,2)) sv];
                Sw=[Sw X_curr'*(Lv'*lw+(V_curr'*C')*(Bh'*[wr,wi]));
                    zeros(2*m,size(Sw,2)) sw];
                Lv=[Lv lv]; Lw=[Lw lw];
            end
        end
        
        bhr=[wr,wi]'*Bh; cr=C*[vr,vi];

        x=compute_x_2(ar,ai,bk,bk1,bhr(1:m,:),...
                                bhr(m+1:2*m,:),cr(:,1:m),cr(:,m+1:2*m),Im);
        K=K+(E*[vr,vi])*x*bhr; Kh=Kh+cr*x*([wr,wi]'*Eh);

        block_size = 2 * m;
        idx_range = col_idx : col_idx + block_size - 1;
        V(:, idx_range) = [vr, vi];
        W(:, idx_range) = [wr, wi];
        X(idx_range, idx_range) = x;
        B_=B_+(E*[vr,vi])*x*[Im; 0*Im]; C_=C_+[Im,0*Im]*x*([wr,wi]'*Eh);
        res=max(sqrt(eig(full(B_'*B_*C_*C_'))))/...
                                         max(sqrt(eig(full(B'*B*Ch*Ch'))));
        Residual=res
        Res=[Res res];
        col_idx = col_idx + block_size;
        k = k + 2;
    
    elseif k<length(a) && isreal(ak) && isreal(ak1) && ~isreal(bk)

        disp('Case 4')

        vk1 = solve_shifted_smw(A, E, K, C, E*vk, p, ak1);
        
        br = real(bk); bi = imag(bk);
        vr = vk; vi = vk1;
        wr = real(wk); wi = imag(wk);
        
        if flag_s
            sv=kron([-ak 1; 0 -ak1],Im); lv=kron([-1 0],Im);
            sw=kron(-[br bi; -bi br],Im); lw=kron([-1 0],Im);
            if k==1
                Sv=blkdiag(Sv,sv); Lv=[Lv lv];
                Sw=blkdiag(Sw,sw); Lw=[Lw lw];
            else
                Sv=[Sv X_curr*(Lw'*lv+(W_curr'*Bh)*(C*[vr,vi]));
                    zeros(2*m,size(Sv,2)) sv];
                Sw=[Sw X_curr'*(Lv'*lw+(V_curr'*C')*(Bh'*[wr,wi]));
                    zeros(2*m,size(Sw,2)) sw];
                Lv=[Lv lv]; Lw=[Lw lw];
            end
        end
        
        bhr=[wr,wi]'*Bh; cr=C*[vr,vi];

        x=compute_x_3(ak,ak1,br,bi,bhr(1:m,:),...
                                bhr(m+1:2*m,:),cr(:,1:m),cr(:,m+1:2*m),Im);
        K=K+(E*[vr,vi])*x*bhr; Kh=Kh+cr*x*([wr,wi]'*Eh);

        block_size = 2 * m;
        idx_range = col_idx : col_idx + block_size - 1;
        V(:, idx_range) = [vr, vi];
        W(:, idx_range) = [wr, wi];
        X(idx_range, idx_range) = x;
        B_=B_+(E*[vr,vi])*x*[Im; 0*Im]; C_=C_+[Im,0*Im]*x*([wr,wi]'*Eh);
        res=max(sqrt(eig(full(B_'*B_*C_*C_'))))/...
                                         max(sqrt(eig(full(B'*B*Ch*Ch'))));
        Residual=res
        Res=[Res res];
        col_idx = col_idx + block_size;
        k = k + 2;

    end
    
    if k>=length(a)

        if mod(c_alt,2) == 1
            if size(V_proj,2) >= rank_max
                r_idx_v=1;
            end
            actual_cols = col_idx - 1;
            last_cols = actual_cols-r_idx_v*(m);
            V_proj=V(:,last_cols+1:actual_cols);
            [vp,~]=qr(V_proj,0);
            er1=vp'*E*vp; ar1=vp'*A*vp; br1=vp'*B_;
            a_new = eig_sort(er1\ar1,er1\br1,(er1\br1)');
            a_new=(-abs(real(a_new))+sqrt(-1).*imag(a_new)).';
            if abs(imag(a_new(1))) <= 1e-8
                a_new=real(a_new(1));
            else
                a_new=[a_new(1) conj(a_new(1))];
            end
            a=[a a_new]; b=[b a_new];
            r_idx_v=r_idx_v+1;


        else
            
            if size(W_proj,2) >= rank_max
                r_idx_w=1;
            end
            actual_cols = col_idx - 1;
            last_cols = actual_cols-r_idx_w*(m);
            W_proj=W(:,last_cols+1:actual_cols);
            [wp,~]=qr(W_proj,0);
            er2=wp'*Eh*wp; ar2=wp'*Ah*wp; cr2=C_*wp;
            b_new = eig_sort(ar2/er2,(cr2/er2)',cr2/er2);
            b_new=(-abs(real(b_new))+sqrt(-1).*imag(b_new)).';
            if abs(imag(b_new(1))) <= 1e-8
                b_new=real(b_new(1));
            else
                b_new=[b_new(1) conj(b_new(1))];
            end
            b=[b b_new]; a=[a b_new];
            r_idx_w=r_idx_w+1;

        end

    end

    itr=itr+1;

end

%% Trimming
actual_cols = col_idx - 1;
V = V(:, 1:actual_cols);
W = W(:, 1:actual_cols);
X = X(1:actual_cols, 1:actual_cols);
a_used=a(1:(actual_cols)/(m));
b_used=b(1:(actual_cols)/(m));

%% Helper Functions

function vk = solve_shifted_smw(A, E, K, C, B_, p, ak)
    M0 = A + ak * E;
    RHS = [B_, K];
    XY  = M0 \ RHS;
    nB = size(B_, 2);
    Xs  = XY(:, 1:nB);
    Y  = XY(:, nB+1:end);
    S = eye(p) - C * Y;
    Ws = S \ (C * Xs);
    vk = Xs + Y * Ws;
end
    
function [x] = compute_x_1(ar,ai,br,bi,bhr,bhi,cr,ci,Im)
    s = ar + br; d_r = s^2 + ai^2 - bi^2; d_i = 2*s*bi;
    Delta_d = d_r^2 + d_i^2;
    n_1_r = -s*Im - (s*bhr - bi*bhi)*cr - ai*bhr*ci;
    n_1_i = -bi*Im - (s*bhi + bi*bhr)*cr - ai*bhi*ci;
    n_2_r =  ai*Im + ai*bhr*cr - (s*bhr - bi*bhi)*ci;
    n_2_i =  ai*bhi*cr - (s*bhi + bi*bhr)*ci;
    x1 = (n_1_r*d_r + n_1_i*d_i) / Delta_d;
    x2 = (n_2_r*d_r + n_2_i*d_i) / Delta_d;
    x3 = (n_1_i*d_r - n_1_r*d_i) / Delta_d;
    x4 = (n_2_i*d_r - n_2_r*d_i) / Delta_d;
    x = inv([x1 x2; x3 x4]);
end

function [x] = compute_x_2(ar,ai,b,b1,bhr,bhi,cr,ci,Im)
    alpha_a = ar + b; beta_a  = ar + b1; d_a = alpha_a^2 + ai^2;
    e_a = beta_a^2 + ai^2;
    x1 = -(alpha_a/d_a)*Im - (alpha_a/d_a)*(bhr * cr) - (ai/d_a)*(bhr * ci);
    x2 =  (ai/d_a)*Im + (ai/d_a)*(bhr * cr) - (alpha_a/d_a)*(bhr * ci);
    x3 = ((ai^2 - alpha_a*beta_a)/(d_a*e_a))*Im +...
                        ((ai^2 - alpha_a*beta_a)/(d_a*e_a))*(bhr * cr) ...
                         - (ai*(alpha_a+beta_a)/(d_a*e_a))*(bhr * ci) -...
                             (beta_a/e_a)*(bhi * cr) - (ai/e_a)*(bhi * ci);
    x4 = (ai*(alpha_a+beta_a)/(d_a*e_a))*Im +...
                             (ai*(alpha_a+beta_a)/(d_a*e_a))*(bhr * cr) ...
                      + ((ai^2 - alpha_a*beta_a)/(d_a*e_a))*(bhr * ci) +...
                             (ai/e_a)*(bhi * cr) - (beta_a/e_a)*(bhi * ci);
    x =inv([x1 x2; x3 x4]);
end

function [x] = compute_x_3(a,a1,br,bi,bhr,bhi,cr,ci,Im)
    alpha_b = a + br; beta_b = a1 + br; d_b = alpha_b^2 + bi^2;
    e_b = beta_b^2 + bi^2;
    x1 = -(alpha_b/d_b)*Im - (alpha_b/d_b)*(bhr*cr) - (bi/d_b)*(bhi*cr);
    x2 = ((bi^2 - alpha_b*beta_b)/(d_b*e_b))*Im + ...
                          ((bi^2 - alpha_b*beta_b)/(d_b*e_b))*(bhr*cr)- ...
       (bi*(alpha_b+beta_b)/(d_b*e_b))*(bhi*cr) -...
                                 (beta_b/e_b)*(bhr*ci) - (bi/e_b)*(bhi*ci);
    x3 =  (bi/d_b)*Im + (bi/d_b)*(bhr*cr) - (alpha_b/d_b)*(bhi*cr);
    x4 = (bi*(alpha_b+beta_b)/(d_b*e_b))*Im +...
                         (bi*(alpha_b+beta_b)/(d_b*e_b))*(bhr*cr) +...
                          ((bi^2 - alpha_b*beta_b)/(d_b*e_b))*(bhi*cr) +...
                                 (bi/e_b)*(bhr*ci) - (beta_b/e_b)*(bhi*ci);
    x=inv([x1 x2; x3 x4]);
end

function [sig] = eig_sort(A,B,C)
A=full(A); B=full(B); C=full(C);
[ve,se] = eig(A); se = diag(se); weT = inv(ve);
re=length(se); quan = zeros(re,1);

for ke = 1:re  
     if  (norm(C*ve(:,ke)) ~= 0)  &&  (norm(weT(ke,:)*B) ~= 0)
         quan(ke) = norm(C*ve(:,ke))*norm(weT(ke,:)*B)/abs(real(se(ke)));
     else
         quan(ke) = 0;
     end
end
[~,inds] = sort(quan,'descend'); sig = ( se(inds) ).';
end


end
\end{verbatim}
\section*{Appendix B: MATLAB Implementation of UN-RADI}
\begin{verbatim}
function [V,W,Tv,Tw,X,K,Kh,Res,a_used,b_used] = ...
               UN_RADI(E,A,B,C,Eh,Ah,Bh,Ch,a,b,a_in,b_in,tol,kmax,rank_max)

% Unified Low-rank ADI Algorithm for Large-scale Non-symmetric Algebraic
% Riccati Equations (NARE)
% It solves the NARE 
% A Q Eh + E Q Ah - E Q Bh C Q Eh + B Ch =0
% Q \approx V Tv X Tw' W'
% K \approx K_gain = E Q Bh
% Kh \approx Kh_gain = C Q Eh
% Author: Umair Zulfiqar
% Date: 20th April 2026

if any(real(a) >= -1e-8) || any(real(a_in) >= -1e-8)
    disp('Error: All the ADI shifts must have negative real parts')
    disp('Fix it and try again!')
    return
end


%% Initialization
[n,m]=size(B); p=size(C,1); nh=size(Bh,1); Im=eye(m);

if ~isempty(a) || ~isempty(b)
    kmax = length(a);
else
    a=a_in; b=b_in;
end

max_cols = (kmax+2) * m;
V = zeros(n, max_cols);
Tv = zeros(max_cols, max_cols);
W = zeros(nh, max_cols);
Tw = zeros(max_cols, max_cols);
V_proj=[]; W_proj=[];
X = zeros(max_cols, max_cols);
Sv = zeros(max_cols, max_cols); Lv = zeros(m, max_cols);
Sw = zeros(max_cols, max_cols); Lw = zeros(m, max_cols);
B_ = B; Bv_ = B; C_ = Ch; Cw_ = Ch;
col_idx = 1;
k = 1; r_idx_v = 1; r_idx_w = 1;
K=zeros(n,p); Kh=zeros(p,nh);
Bhr_lyap = zeros(max_cols,p); Bhr_nare = zeros(max_cols,p);
Br_lyap = zeros(max_cols,m); Br_nare = zeros(max_cols,m);
Cr_lyap  = zeros(p,max_cols); Cr_nare  = zeros(p,max_cols);
Chr_lyap  = zeros(m,max_cols); Chr_nare  = zeros(m,max_cols);
flag=1; res=1; Res=[]; itr=1; c_alt=1;
%% Iterations

while flag
    
    if itr>k
        disp('ADI Shifts are not ordered properly')
        flag=0;
        break
    end
    
    Iteration_No=itr

    if k>kmax
        disp('Maximum iteration Reached')
        flag=0;
        break
    end

    if res<tol
        disp('UN-RADI Converged')
        flag=0;
        break
    end
   
    ak = a(k); bk = b(k);

    if k<length(a)
        ak1=a(k+1); bk1=b(k+1);
    end

    gm_v=sqrt(-2*real(ak)); gm_w=sqrt(-2*real(bk));

    vk=(A+ak*E)\Bv_; wk=(Ah'+bk*Eh')\(Cw_');
    
    if isreal(ak) && isreal(bk)

        disp('Case 1')
        
        sv=-ak*Im; lv=-gm_v*Im; sw=-bk*Im; lw=-gm_w*Im;
        Bv_=Bv_+gm_v*gm_v*(E*vk); Cw_=Cw_+gm_w*gm_w*(wk'*Eh);

        block_size = m;
        idx_range = col_idx : col_idx + block_size - 1;
        
        if k==1
            Sv(idx_range,idx_range)=sv; Sw(idx_range,idx_range)=sw;
        else
            Sv(1:idx_range(end), idx_range) = [Lv(:,1:col_idx-1)'*lv; sv];
            Sw(1:idx_range(end), idx_range) = [Lw(:,1:col_idx-1)'*lw; sw];
        end
        Lv(:,idx_range)=lv; Lw(:,idx_range)=lw;
        Bhr_lyap(idx_range,:)=(gm_w*wk)'*Bh;
        Cr_lyap(:,idx_range)=C*(gm_v*vk);

        Svc=Sv(1:idx_range(end),1:idx_range(end)); I=eye(idx_range(end));
        Swc=Sw(1:idx_range(end),1:idx_range(end));
        Lvc=Lv(:,1:idx_range(end)); Lwc=Lw(:,1:idx_range(end));

        if k==1
            tv=(-Svc'+ak*I)\(Lvc');
            tw=(-Swc'+bk*I)\(Lwc');
        else
            Tvc=Tv(1:idx_range(end),1:col_idx-1);
            Twc=Tw(1:idx_range(end),1:col_idx-1);
            Xc=X(1:col_idx-1,1:col_idx-1);
            Bc_nare=Br_nare(1:col_idx-1,:);
            Bhc_nare=Bhr_nare(1:col_idx-1,:);
            Chc_nare=Chr_nare(:,1:col_idx-1);
            Cc_nare=Cr_nare(:,1:col_idx-1);
            Bhc_lyap=Bhr_lyap(1:idx_range(end),:);
            Cc_lyap=Cr_lyap(:,1:idx_range(end));
            tv=(-Svc'-Tvc*Xc*Bhc_nare*Cc_lyap+ak*I)\(Lvc'-Tvc*Bc_nare);
            tw=(-Swc'-Twc*Xc'*Cc_nare'*Bhc_lyap'+bk*I)\(Lwc'-Twc*Chc_nare');
        end
        Tv(1:idx_range(end), idx_range) = tv;
        Tw(1:idx_range(end), idx_range) = tw;

        V(:, idx_range) = gm_v*vk;
        W(:, idx_range) = gm_w*wk;
        
        v_nare=V(:,1:idx_range(end))*tv; w_nare=W(:,1:idx_range(end))*tw;
        bhr=w_nare'*Bh; cr=C*v_nare;
        Bhr_nare(idx_range,:)=bhr; Cr_nare(:,idx_range)=cr;
        x=-(ak+bk)*inv(Im+bhr*cr);
        K=K+(E*v_nare)*x*bhr; Kh=Kh+cr*x*(w_nare'*Eh);
        Br_nare(idx_range,:)=-x; Chr_nare(:,idx_range)=-x;

        
        X(idx_range, idx_range) = x;
        B_=B_+(E*v_nare)*x; C_=C_+x*(w_nare'*Eh);
        res=max(sqrt(eig(full(B_'*B_*C_*C_'))))/...
                                         max(sqrt(eig(full(B'*B*Ch*Ch'))));
        Residual=res
        Res=[Res res];
        col_idx = col_idx + block_size;
        k = k + 1;
        
    elseif ~isreal(ak) && ~isreal(bk)

        disp('Case 2')

        ar = real(ak); ai = imag(ak); del_v=ar/ai;
        bt_v=sqrt(1+del_v^2)/(2*del_v);
        br = real(bk); bi = imag(bk); del_w=br/bi;
        bt_w=sqrt(1+del_w^2)/(2*del_w);
        vr=sqrt(2)*gm_v*(real(vk)+del_v*imag(vk));
        vi=sqrt(2)*gm_v*sqrt((del_v^2)+1)*imag(vk);
        Bv_=Bv_+2*gm_v*gm_v*(E*(real(vk)+del_v*imag(vk)));
        wr=sqrt(2)*gm_w*(real(wk)+del_w*imag(wk));
        wi=sqrt(2)*gm_w*sqrt((del_w^2)+1)*imag(wk);
        Cw_=Cw_+2*gm_w*gm_w*((real(wk)+del_w*imag(wk))'*Eh);
        
        sv=kron(gm_v*gm_v*[1 bt_v; -bt_v 0],Im);
        lv=kron(-sqrt(2)*gm_v*[1 0],Im);
        sw=kron(gm_w*gm_w*[1 bt_w; -bt_w 0],Im);
        lw=kron(-sqrt(2)*gm_w*[1 0],Im);

        block_size = 2 * m;
        idx_range = col_idx : col_idx + block_size - 1;
        
        if k==1
            Sv(idx_range,idx_range)=sv; Sw(idx_range,idx_range)=sw;
        else
            Sv(1:idx_range(end), idx_range) = [Lv(:,1:col_idx-1)'*lv; sv];
            Sw(1:idx_range(end), idx_range) = [Lw(:,1:col_idx-1)'*lw; sw];
        end
        Lv(:,idx_range)=lv; Lw(:,idx_range)=lw;
        Bhr_lyap(idx_range,:)=[wr,wi]'*Bh; Cr_lyap(:,idx_range)=C*[vr,vr];

        Svc=Sv(1:idx_range(end),1:idx_range(end)); I=eye(idx_range(end));
        Swc=Sw(1:idx_range(end),1:idx_range(end));
        Lvc=Lv(:,1:idx_range(end)); Lwc=Lw(:,1:idx_range(end));

        if k==1
            tv=(-Svc'+ak*I)\(Lvc');
            tw=(-Swc'+bk*I)\(Lwc');
        else
            Tvc=Tv(1:idx_range(end),1:col_idx-1);
            Twc=Tw(1:idx_range(end),1:col_idx-1);
            Xc=X(1:col_idx-1,1:col_idx-1);
            Bc_nare=Br_nare(1:col_idx-1,:);
            Bhc_nare=Bhr_nare(1:col_idx-1,:);
            Chc_nare=Chr_nare(:,1:col_idx-1);
            Cc_nare=Cr_nare(:,1:col_idx-1);
            Bhc_lyap=Bhr_lyap(1:idx_range(end),:);
            Cc_lyap=Cr_lyap(:,1:idx_range(end));
            tv=(-Svc'-Tvc*Xc*Bhc_nare*Cc_lyap+ak*I)\(Lvc'-Tvc*Bc_nare);
            tw=(-Swc'-Twc*Xc'*Cc_nare'*Bhc_lyap'+bk*I)\(Lwc'-Twc*Chc_nare');
        end
        tvr=real(tv); tvi=imag(tv); twr=real(tw); twi=imag(tw);
        Tv(1:idx_range(end), idx_range) = [tvr,tvi];
        Tw(1:idx_range(end), idx_range) = [twr,twi];

        V(:, idx_range) = [vr, vi];
        W(:, idx_range) = [wr, wi];
        
        v_nare=V(:,1:idx_range(end))*[tvr,tvi];
        w_nare=W(:,1:idx_range(end))*[twr,twi];
        bhr=w_nare'*Bh; cr=C*v_nare;
        Bhr_nare(idx_range,:)=bhr; Cr_nare(:,idx_range)=cr;
        x = compute_x_1(ar,ai,br,bi,bhr(1:m,:),...
                                bhr(m+1:2*m,:),cr(:,1:m),cr(:,m+1:2*m),Im);
        K=K+(E*v_nare)*x*bhr; Kh=Kh+cr*x*(w_nare'*Eh);
        Br_nare(idx_range,:)=-x*[Im 0*Im]';
        Chr_nare(:,idx_range)=-[Im 0*Im]*x;

        
        X(idx_range, idx_range) = x;
        B_=B_+(E*v_nare)*x*[Im; 0*Im]; C_=C_+[Im,0*Im]*x*(w_nare'*Eh);
        res=max(sqrt(eig(full(B_'*B_*C_*C_'))))/...
                                         max(sqrt(eig(full(B'*B*Ch*Ch'))));
        Residual=res
        Res=[Res res];
        col_idx = col_idx + block_size;
        k = k + 2;

    elseif k<length(a) && ~isreal(ak) && isreal(bk) && isreal(bk1)

        disp('Case 3')

        sw=-bk*Im; lw=-gm_w*Im; Cw_=Cw_+gm_w*gm_w*(wk'*Eh);

        block_size = m;
        idx_range = col_idx : col_idx + block_size - 1;
        idx_range_v = col_idx : col_idx + (2*block_size) - 1;

        if k==1
            Sw(idx_range,idx_range)=sw;
        else
            Sw(1:idx_range(end), idx_range) = [Lw(:,1:col_idx-1)'*lw; sw];
        end
        Lw(:,idx_range)=lw;
        Bhr_lyap(idx_range,:)=(gm_w*wk)'*Bh;
        
        gm_w1=sqrt(-2*real(bk1)); sw=-bk1*Im; lw=-gm_w1*Im;
        wk1=(Ah'+bk1*Eh')\(Cw_'); Cw_=Cw_+gm_w1*gm_w1*(wk1'*Eh);

        block_size = m;
        idx_range = col_idx + block_size: col_idx + (2*block_size) - 1;
        
        Sw(1:idx_range(end), idx_range) = [Lw(:,1:(col_idx+m)-1)'*lw; sw];
        Lw(:,idx_range)=lw; Bhr_lyap(idx_range,:)=(gm_w1*wk1)'*Bh;

        ar = real(ak); ai = imag(ak); del_v=ar/ai;
        bt_v=sqrt(1+del_v^2)/(2*del_v);
        vr=sqrt(2)*gm_v*(real(vk)+del_v*imag(vk));
        vi=sqrt(2)*gm_v*sqrt((del_v^2)+1)*imag(vk);
        Bv_=Bv_+2*gm_v*gm_v*(E*(real(vk)+del_v*imag(vk)));

        sv=kron(gm_v*gm_v*[1 bt_v; -bt_v 0],Im);
        lv=kron(-sqrt(2)*gm_v*[1 0],Im);
        
        if k==1
            Sv(idx_range_v,idx_range_v)=sv;
        else
            Sv(1:idx_range_v(end), idx_range_v) = [Lv(:,1:col_idx-1)'*lv; sv];
        end
        Lv(:,idx_range_v)=lv; Cr_lyap(:,idx_range_v)=C*[vr,vi];

        Svc=Sv(1:idx_range(end),1:idx_range(end)); I=eye(idx_range(end));
        Swc=Sw(1:idx_range(end),1:idx_range(end));
        Lvc=Lv(:,1:idx_range(end)); Lwc=Lw(:,1:idx_range(end));

        if k==1
            tv=(-Svc'+ak*I)\(Lvc');
            tw=(-Swc'+bk*I)\(Lwc');
            tw1=(-Swc'+bk1*I)\tw;
        else
            Tvc=Tv(1:idx_range(end),1:col_idx-1);
            Twc=Tw(1:idx_range(end),1:col_idx-1);
            Xc=X(1:col_idx-1,1:col_idx-1);
            Bc_nare=Br_nare(1:col_idx-1,:);
            Bhc_nare=Bhr_nare(1:col_idx-1,:);
            Chc_nare=Chr_nare(:,1:col_idx-1);
            Cc_nare=Cr_nare(:,1:col_idx-1);
            Bhc_lyap=Bhr_lyap(1:idx_range(end),:);
            Cc_lyap=Cr_lyap(:,1:idx_range(end));
            tv=(-Svc'-Tvc*Xc*Bhc_nare*Cc_lyap+ak*I)\(Lvc'-Tvc*Bc_nare);
            tw=(-Swc'-Twc*Xc'*Cc_nare'*Bhc_lyap'+bk*I)\(Lwc'-Twc*Chc_nare');
            tw1=(-Swc'-Twc*Xc'*Cc_nare'*Bhc_lyap'+bk1*I)\tw;        
        end
        
        tvr=real(tv); tvi=imag(tv); twr=tw; twi=tw1;
        Tv(1:idx_range(end), idx_range_v) = [tvr,tvi];
        Tw(1:idx_range(end), idx_range_v) = [twr,twi];

        V(:, idx_range_v) = [vr, vi];
        W(:, idx_range_v) = [gm_w*wk, gm_w1*wk1];
        
        v_nare=V(:,1:idx_range(end))*[tvr,tvi];
        w_nare=W(:,1:idx_range(end))*[twr,twi];
        
        bhr=w_nare'*Bh; cr=C*v_nare;
        Bhr_nare(idx_range_v,:)=bhr; Cr_nare(:,idx_range_v)=cr;
        x = compute_x_2(ar,ai,bk,bk1,bhr(1:m,:),...
                                bhr(m+1:2*m,:),cr(:,1:m),cr(:,m+1:2*m),Im);
        K=K+(E*v_nare)*x*bhr; Kh=Kh+cr*x*(w_nare'*Eh);
        Br_nare(idx_range_v,:)=-x*[Im 0*Im]';
        Chr_nare(:,idx_range_v)=-[Im 0*Im]*x;

        X(idx_range_v, idx_range_v) = x;
        B_=B_+(E*v_nare)*x*[Im; 0*Im]; C_=C_+[Im,0*Im]*x*(w_nare'*Eh);
        res=max(sqrt(eig(full(B_'*B_*C_*C_'))))/...
                                         max(sqrt(eig(full(B'*B*Ch*Ch'))));
        Residual=res
        Res=[Res res];
        col_idx = col_idx + 2*block_size;
        k = k + 2;
    
    elseif k<length(a) && isreal(ak) && isreal(ak1) && ~isreal(bk)

        disp('Case 4')

        sv=-ak*Im; lv=-gm_v*Im; Bv_=Bv_+gm_v*gm_v*(E*vk);

        block_size = m;
        idx_range = col_idx : col_idx + block_size - 1;
        idx_range_w = col_idx : col_idx + (2*block_size) - 1;
        
        if k==1
            Sv(idx_range,idx_range)=sv;
        else
            Sv(1:idx_range(end), idx_range) = [Lv(:,1:col_idx-1)'*lv; sv];
        end
        Lv(:,idx_range)=lv; Cr_lyap(:,idx_range)=C*(gm_v*vk);

        gm_v1=sqrt(-2*real(ak1)); sv=-ak1*Im; lv=-gm_v1*Im;
        vk1=(A+ak1*E)\Bv_; Bv_=Bv_+gm_v1*gm_v1*(E*vk1);

        block_size = m;
        idx_range = col_idx+m : col_idx + (2*block_size) - 1;
        
        Sv(1:idx_range(end), idx_range) = [Lv(:,1:(col_idx+m)-1)'*lv; sv];
        Lv(:,idx_range)=lv; Cr_lyap(:,idx_range)=C*(gm_v1*vk1);

        br = real(bk); bi = imag(bk); del_w=br/bi;
        bt_w=sqrt(1+del_w^2)/(2*del_w);
        wr=sqrt(2)*gm_w*(real(wk)+del_w*imag(wk));
        wi=sqrt(2)*gm_w*sqrt((del_w^2)+1)*imag(wk);
        Cw_=Cw_+2*gm_w*gm_w*((real(wk)+del_w*imag(wk))'*Eh);
        
        sw=kron(gm_w*gm_w*[1 bt_w; -bt_w 0],Im);
        lw=kron(-sqrt(2)*gm_w*[1 0],Im);

        if k==1
            Sw(idx_range_w,idx_range_w)=sw;
        else
            Sw(1:idx_range(end), idx_range_w) = [Lw(:,1:col_idx-1)'*lw; sw];
        end
        Lw(:,idx_range_w)=lw; Bhr_lyap(idx_range_w,:)=[wr,wi]'*Bh;

        Svc=Sv(1:idx_range(end),1:idx_range(end)); I=eye(idx_range(end));
        Swc=Sw(1:idx_range(end),1:idx_range(end));
        Lvc=Lv(:,1:idx_range(end)); Lwc=Lw(:,1:idx_range(end));

        if k==1
            tv=(-Svc'+ak*I)\(Lvc');
            tv1=(-Svc'+ak1*I)\tv;
            tw=(-Swc'+bk*I)\(Lwc');
        else
            Tvc=Tv(1:idx_range(end),1:col_idx-1);
            Twc=Tw(1:idx_range(end),1:col_idx-1);
            Xc=X(1:col_idx-1,1:col_idx-1);
            Bc_nare=Br_nare(1:col_idx-1,:);
            Bhc_nare=Bhr_nare(1:col_idx-1,:);
            Chc_nare=Chr_nare(:,1:col_idx-1);
            Cc_nare=Cr_nare(:,1:col_idx-1);
            Bhc_lyap=Bhr_lyap(1:idx_range(end),:);
            Cc_lyap=Cr_lyap(:,1:idx_range(end));
            tv=(-Svc'-Tvc*Xc*Bhc_nare*Cc_lyap+ak*I)\(Lvc'-Tvc*Bc_nare);
            tv1=(-Svc'-Tvc*Xc*Bhc_nare*Cc_lyap+ak1*I)\tv;
            tw=(-Swc'-Twc*Xc'*Cc_nare'*Bhc_lyap'+bk*I)\(Lwc'-Twc*Chc_nare');
        end
      
        tvr=tv; tvi=tv1; twr=real(tw); twi=imag(tw);
        Tv(1:idx_range(end), idx_range_w) = [tvr,tvi];
        Tw(1:idx_range(end), idx_range_w) = [twr,twi];

        V(:, idx_range_w) = [gm_v*vk, gm_v1*vk1];
        W(:, idx_range_w) = [wr, wi];
        
        v_nare=V(:,1:idx_range(end))*[tvr,tvi];
        w_nare=W(:,1:idx_range(end))*[twr,twi];

        bhr=w_nare'*Bh; cr=C*v_nare;
        Bhr_nare(idx_range_w,:)=bhr; Cr_nare(:,idx_range_w)=cr;
        x = compute_x_3(ak,ak1,br,bi,bhr(1:m,:),...
                                bhr(m+1:2*m,:),cr(:,1:m),cr(:,m+1:2*m),Im);
        K=K+(E*v_nare)*x*bhr; Kh=Kh+cr*x*(w_nare'*Eh);
        Br_nare(idx_range_w,:)=-x*[Im 0*Im]';
        Chr_nare(:,idx_range_w)=-[Im 0*Im]*x;
        
        X(idx_range_w, idx_range_w) = x;
        B_=B_+(E*v_nare)*x*[Im; 0*Im]; C_=C_+[Im,0*Im]*x*(w_nare'*Eh);
        res=max(sqrt(eig(full(B_'*B_*C_*C_'))))/...
                                        max(sqrt(eig(full(B'*B*Ch*Ch'))));
        Residual=res
        Res=[Res res];
        col_idx = col_idx + 2*block_size;
        k = k + 2;

    end
    
    if k>=length(a)

        if mod(c_alt,2) == 1
            if size(V_proj,2) >= rank_max
                r_idx_v=1;
            end
            actual_cols = col_idx - 1;
            last_cols = actual_cols-r_idx_v*(m);
            V_proj=V(:,1:actual_cols)*Tv(1:actual_cols,...
                                                  last_cols+1:actual_cols);
            [vp,~]=qr(V_proj,0);
            er1=vp'*E*vp; ar1=vp'*A*vp; br1=vp'*B_;
            a_new = eig_sort(er1\ar1,er1\br1,(er1\br1)');
            a_new=(-abs(real(a_new))+sqrt(-1).*imag(a_new)).';
            if abs(imag(a_new(1))) <= 1e-8
                a_new=real(a_new(1));
            else
                a_new=[a_new(1) conj(a_new(1))];
            end
            a=[a a_new]; b=[b a_new];
            r_idx_v=r_idx_v+1;


        else
            
            if size(W_proj,2) >= rank_max
                r_idx_w=1;
            end
            actual_cols = col_idx - 1;
            last_cols = actual_cols-r_idx_w*(m);
            W_proj=W(:,1:actual_cols)*Tw(1:actual_cols,...
                                                  last_cols+1:actual_cols);
            [wp,~]=qr(W_proj,0);
            er2=wp'*Eh*wp; ar2=wp'*Ah*wp; cr2=C_*wp;
            b_new = eig_sort(ar2/er2,(cr2/er2)',cr2/er2);
            b_new=(-abs(real(b_new))+sqrt(-1).*imag(b_new)).';
            if abs(imag(b_new(1))) <= 1e-8
                b_new=real(b_new(1));
            else
                b_new=[b_new(1) conj(b_new(1))];
            end
            b=[b b_new]; a=[a b_new];
            r_idx_w=r_idx_w+1;

        end

    end

    itr=itr+1;

end

%% Trimming
actual_cols = col_idx - 1;
V = V(:, 1:actual_cols);
W = W(:, 1:actual_cols);
Tv = Tv(1:actual_cols, 1:actual_cols);
Tw = Tw(1:actual_cols, 1:actual_cols);
X = X(1:actual_cols, 1:actual_cols);
a_used=a(1:(actual_cols)/(m));
b_used=b(1:(actual_cols)/(m));

%% Helper Functions

function [x] = compute_x_1(ar,ai,br,bi,bhr,bhi,cr,ci,Im)
    s = ar + br; d_r = s^2 + ai^2 - bi^2; d_i = 2*s*bi;
    Delta_d = d_r^2 + d_i^2;
    n_1_r = -s*Im - (s*bhr - bi*bhi)*cr - ai*bhr*ci;
    n_1_i = -bi*Im - (s*bhi + bi*bhr)*cr - ai*bhi*ci;
    n_2_r =  ai*Im + ai*bhr*cr - (s*bhr - bi*bhi)*ci;
    n_2_i =  ai*bhi*cr - (s*bhi + bi*bhr)*ci;
    x1 = (n_1_r*d_r + n_1_i*d_i) / Delta_d;
    x2 = (n_2_r*d_r + n_2_i*d_i) / Delta_d;
    x3 = (n_1_i*d_r - n_1_r*d_i) / Delta_d;
    x4 = (n_2_i*d_r - n_2_r*d_i) / Delta_d;
    x = inv([x1 x2; x3 x4]);
end

function [x] = compute_x_2(ar,ai,b,b1,bhr,bhi,cr,ci,Im)
    alpha_a = ar + b; beta_a  = ar + b1; d_a = alpha_a^2 + ai^2;
    e_a = beta_a^2 + ai^2;
    x1 = -(alpha_a/d_a)*Im - (alpha_a/d_a)*(bhr * cr) - (ai/d_a)*(bhr * ci);
    x2 =  (ai/d_a)*Im + (ai/d_a)*(bhr * cr) - (alpha_a/d_a)*(bhr * ci);
    x3 = ((ai^2 - alpha_a*beta_a)/(d_a*e_a))*Im +...
                        ((ai^2 - alpha_a*beta_a)/(d_a*e_a))*(bhr * cr) ...
                         - (ai*(alpha_a+beta_a)/(d_a*e_a))*(bhr * ci) -...
                             (beta_a/e_a)*(bhi * cr) - (ai/e_a)*(bhi * ci);
    x4 = (ai*(alpha_a+beta_a)/(d_a*e_a))*Im +...
                             (ai*(alpha_a+beta_a)/(d_a*e_a))*(bhr * cr) ...
                      + ((ai^2 - alpha_a*beta_a)/(d_a*e_a))*(bhr * ci) +...
                             (ai/e_a)*(bhi * cr) - (beta_a/e_a)*(bhi * ci);
    x =inv([x1 x2; x3 x4]);
end

function [x] = compute_x_3(a,a1,br,bi,bhr,bhi,cr,ci,Im)
    alpha_b = a + br; beta_b = a1 + br; d_b = alpha_b^2 + bi^2;
    e_b = beta_b^2 + bi^2;
    x1 = -(alpha_b/d_b)*Im - (alpha_b/d_b)*(bhr*cr) - (bi/d_b)*(bhi*cr);
    x2 = ((bi^2 - alpha_b*beta_b)/(d_b*e_b))*Im + ...
                          ((bi^2 - alpha_b*beta_b)/(d_b*e_b))*(bhr*cr)- ...
       (bi*(alpha_b+beta_b)/(d_b*e_b))*(bhi*cr) -...
                                 (beta_b/e_b)*(bhr*ci) - (bi/e_b)*(bhi*ci);
    x3 =  (bi/d_b)*Im + (bi/d_b)*(bhr*cr) - (alpha_b/d_b)*(bhi*cr);
    x4 = (bi*(alpha_b+beta_b)/(d_b*e_b))*Im +...
                         (bi*(alpha_b+beta_b)/(d_b*e_b))*(bhr*cr) +...
                          ((bi^2 - alpha_b*beta_b)/(d_b*e_b))*(bhi*cr) +...
                                 (bi/e_b)*(bhr*ci) - (beta_b/e_b)*(bhi*ci);
    x=inv([x1 x2; x3 x4]);
end

function [sig] = eig_sort(A,B,C)
A=full(A); B=full(B); C=full(C);
[ve,se] = eig(A); se = diag(se); weT = inv(ve);
re=length(se); quan = zeros(re,1);

for ke = 1:re  
     if  (norm(C*ve(:,ke)) ~= 0)  &&  (norm(weT(ke,:)*B) ~= 0)
         quan(ke) = norm(C*ve(:,ke))*norm(weT(ke,:)*B)/abs(real(se(ke)));
     else
         quan(ke) = 0;
     end
end
[~,inds] = sort(quan,'descend'); sig = ( se(inds) ).';
end

end
\end{verbatim}
%\bibliography{mybibfile}

\end{document}